\documentclass[a4paper,12pt]{amsart}
\usepackage[english]{babel}
\usepackage[T1]{fontenc}
\usepackage{array}
\usepackage{makecell}
\usepackage[a4paper,margin=2.5cm]{geometry}
\usepackage[justification=centering]{caption}

\usepackage{enumerate}
\usepackage{tabularx} % extra features for tabular environment
\usepackage{amsmath}  % improve math presentation
\usepackage{graphicx} % takes care of graphic including machinery

\usepackage{amsmath,amscd,amsbsy,amssymb,latexsym,url,bm,amsthm,mathrsfs}
\usepackage{epsfig,graphicx,subfigure}
\usepackage{enumerate,balance,mathtools}
\usepackage{wrapfig}
\usepackage{mathrsfs,euscript}
\usepackage[usenames]{xcolor}
\usepackage[vlined,boxed,commentsnumbered,linesnumbered,ruled]{algorithm2e}
\usepackage{amsfonts}
\usepackage{amsthm}
\usepackage{fancyhdr}
\usepackage{picinpar}
\usepackage{listings}
\usepackage{layout}
\usepackage[all,cmtip]{xy}
\usepackage{indentfirst}
\usepackage{tikz-cd}
\usepackage{tikz}
\usepackage[all]{xy}
\usepackage{mathrsfs}
\usepackage[colorlinks=true, allcolors=blue]{hyperref}
\usepackage{setspace}

\usepackage{amsmath,amsthm,amssymb,amsxtra,calligra,mathrsfs}
\usepackage{graphicx}
\usepackage{tikz}
\usepackage{tikz-cd} %绘制交换图
\usepackage{xcolor}
\usepackage{upgreek}
\usepackage{comment}
\usepackage{graphicx}
\usepackage{mathtools}
\usepackage{extarrows}
\usepackage{faktor}
\usepackage{mathrsfs,euscript}
\usepackage{comment}
\usepackage{bookmark}
\usepackage{hyperref}
\usepackage{mathrsfs}
\usepackage{multirow}

\hypersetup{
    colorlinks=true, 
    citecolor=blue,
    linkcolor=magenta,
    pdfauthor={},
	bookmarksnumbered=true,
}

\newcommand{\wh}[1]{\widehat{#1}}

\newcommand{\wt}[1]{\widetilde{#1}}

\newcommand{\mb}[1]{\mathbb{#1}}

\newcommand{\ove}[1]{\overline{#1}}

\newcommand{\mtc}[1]{\mathcal{#1}}
\newcommand{\mtf}[1]{\mathfrak{#1}}

\DeclareMathOperator{\coeff}{coeff}
\DeclareMathOperator{\Fut}{Fut}

\DeclareMathOperator{\PGL}{PGL}

\DeclareMathOperator{\Proj}{Proj}

\DeclareMathOperator{\mult}{mult}

\DeclareMathOperator{\Pic}{Pic}

\DeclareMathOperator{\ord}{ord}

\DeclareMathOperator{\Ext}{Ext}
\DeclareMathOperator{\Cl}{Cl}

\DeclareMathOperator{\lct}{lct}

\DeclareMathOperator{\length}{length}

\DeclareMathOperator{\SL}{SL}

\DeclareMathOperator{\GIT}{GIT}

\DeclareMathOperator{\red}{red}
\DeclareMathOperator{\Aut}{Aut}

\DeclareMathOperator{\Ind}{Ind}

\DeclareMathOperator{\vol}{vol}

\DeclareMathOperator{\sm}{sm}
\DeclareMathOperator{\CM}{CM}

\DeclareMathOperator{\Hodge}{Hodge}

\newcommand{\sslash}{\mathbin{\mkern-3mu/\mkern-6mu/\mkern-3mu}}

\newcommand{\sheafHom}{\mathscr{H}\text{\kern -3pt {\calligra\large om}}\,}

\DeclareMathOperator{\Chow}{\mathrm{Chow}}

\newtheorem{theorem}{Theorem}[section]
\newtheorem{lemma}[theorem]{Lemma}
\newtheorem*{convention}{Convention}
\newtheorem{corollary}[theorem]{Corollary}
\newtheorem{prop}[theorem]{Proposition}

\newtheorem{defn}[theorem]{Definition}

\newtheorem{remark}[theorem]{Remark}

\theoremstyle{remark}

\title{Moduli of Genus Six Curves and K-stability}\date{}
\author{Junyan Zhao}
\address{851 S Morgan St, 60607, Chicago, Illinois, USA}
\email{jzhao81@uic.edu}

\begin{document}
\maketitle

\begin{abstract}
The K-moduli theory provides a different compactification of moduli spaces of curves. As a general genus six curve can be canonically embedded into the smooth quintic del Pezzo surface, we study in this paper the K-moduli spaces $\overline{M}^K(c)$ of the quintic log Fano pairs. We classify the strata of genus six curves $C$ appearing in the K-moduli by explicitly describing the wall-crossing structure. The K-moduli spaces interpolate between two birational moduli spaces constructed by GIT and moduli of K3 surfaces via Hodge theory.

\end{abstract}

\tableofcontents

\section{Introduction}

The moduli of curves is a topic of great interest. One of the classical problem is to find different compactifications of the moduli $M_g$ of smooth curves of genus $g$. For instance, people use the Hassett-Keel program, KSBA stability, K-stability, moduli of K3 via Hodge theory, GIT and moduli of boundary polarized log Calabi-Yau pairs (cf. \cite{ADL19,Hac04,HL10,Has99,HH13,ABB23}) to give different birational models of moduli $\ove{M}_3$ of genus $3$ curve, and these different moduli spaces match perfectly to give a commutative diagram of  wall-crossing structure (cf. \cite[Section 9.3.1]{ADL19}). 

The development of K-stability provides a moduli theory for Fano varieties and log Fano pairs, called \emph{K-moduli spaces}. The general K-moduli theory was established by a number of people (cf. \cite{ABHLX20,BLX19,BX19,BHLLX21,CP21,Jia20,LWX21,LXZ22,XZ20}), and it provides us with an approach to give different compactifications of moduli of curves. Our idea is to embed a general curve into some del Pezzo surface $X$ as a divisor, and study the moduli of pairs $(X,cC)$, where $0<c\leq 1$ is some rational coefficient in log Fano region. For example, in the case for $g=3$, we can take $X=\mb{P}^2$ and $0<c<\frac{3}{4}$ (cf. \cite[Section 6]{ADL19}).

This approach highly depends on the geometry of curves. For now, we focus on curves of genus six. A general curve $C$ of genus six has exactly five distinct linear series of degree $6$ and dimension $2$, denoted by $g^2_6$, and $C$ is canonically embedded in the smooth del Pezzo surface $\Sigma_5$ with $(-K_{\Sigma_5})^2=5$ as a divisor in the class $-2K_{\Sigma_5}$ (cf. \cite{ACGH,AH81}). In this paper, we will study the K-(poly/semi)stable log Fano pair $(X,cD)$ which admits a $\mb{Q}$-Gorenstein smoothing to the pair $(\Sigma,cC)$ with $C\in|-2K_{\Sigma}|$ a smooth curve. These K-semistable pairs form an Artin stack $\mtc{M}^K(c)$, whose good moduli space $\ove{M}^K(c)$ is a projective scheme, where $0<c<\frac{1}{2}$ is a rational number. Theses K-moduli spaces are birational to the moduli spaces $\ove{M}_6$ of DM-stable genus six curves.

There are two natural ways to give different birational models of $\ove{M}_6$. Consider the space $${\bf{P}}:=\mb{P}H^0(\Sigma_5,-2K_{\Sigma_5})\simeq\mb{P}^{15}$$ of curves in the class $-2K_{\Sigma_5}$. Since $\Aut(\Sigma_5)\simeq \mathfrak{S}_5$ is a finite group, the GIT moduli space $\ove{M}^{\GIT}:={\bf{P}}/\Aut(X)$ is a natural birational model of moduli space $\ove{M}_6$ of stable curves of genus six. Another observation is that taking the double cover of $\Sigma_5$ along $C$ yields a K3 surface. By studying the period map from $\ove{M}_6$ to the moduli space $\ove{F}^{*}$ of certain lattice-polarized K3 surfaces, one sees that $\ove{M}_6$ is birational to an arithmetic quotient of a bounded symmetric domain of type IV (cf. \cite{AK11}) These two moduli spaces can be related by our K-moduli.

\begin{theorem}\label{72} \textup{(cf. Theorem \ref{55}, \ref{93})} Let $0<c<\frac{1}{2}$ be a rational number.
\begin{enumerate}
    \item For $0<c<1/17$, there is an isomorphism of Artin stacks $\mtc{M}^{K}(c)\simeq {\mtc{M}}^{\GIT}$, which descends to an isomorphism $\ove{M}^{K}(c)\simeq \ove{M}^{\GIT}$ of good moduli spaces. 
    \item For $\frac{11}{28}<c<\frac{1}{2}$ be a rational number and $\mtc{F}^{*}$ be the Kond\={o}'s moduli of K3 surfaces. Then there is a birational morphism $\ove{M}^K(c)\rightarrow \mtc{F}^{*}$, which is an isomorphism in codimension $1$. Moreover, this morphism is the ample model of the Hodge line bundle.
\end{enumerate}

\end{theorem}

Therefore, the K-moduli spaces $\ove{M}^K(c)$ gives rise to an explicit resolution of the birational map $\ove{M}^{\GIT}\dashrightarrow \mtc{F}^{*}$. 

Another question we answer is the image of the birational map from $\ove{M}^K(c)$ to the moduli of curves. Let $M_{\sm}^K(c)$ denotes the open subscheme of $\ove{M}^K(c)$ consisting of the K-polystable pairs $(X,cC)$ with $C$ smooth.

\begin{theorem}\label{7} \textup{(cf. Theorem \ref{52})}
Let $0<c<\frac{1}{2}$ be a rational number, and $\varphi(c):\ove{M}^K(c)\dashrightarrow \ove{M}_6$ be the natural forgetful map defined by $[(X,cC)]\mapsto [C]$.
\begin{enumerate}
    \item The map $\varphi(\varepsilon)$ is restricts to a surjective morphism $$M_{\sm}^K(\varepsilon)\twoheadrightarrow U:=\{\textup{curves with exactly five }g^2_6\}\subseteq M_6.$$
    \item The rational map $\varphi\left(\frac{1}{2}-\varepsilon\right)$ restricts to a surjective morphism $$M_{\sm}^K(1/2-\varepsilon)\twoheadrightarrow M_6\setminus \{\textup{hyperelliptic and bielliptic curves}\}.$$ 
\end{enumerate}
\end{theorem}

In other words, when increasing the coefficient $c$, we gradually add the curves with less than five $g^2_6$, the trigonal curves, and plane quintic curves into our K-moduli spaces. See Section \ref{119} for the introduction to the geometry of genus six curves and the wall-crossing structure of K-moduli spaces.

The explicit wall crossing structure of the K-moduli spaces were studied for $\mb{P}^2,\mb{P}^1\times \mb{P}^1$ and $\mb{P}^3$ in recent years (cf. \cite{ADL19,ADL21,ADL22}). We take $\mb{P}^2$ as an example to briefly illustrate their idea of proof. Consider the K-moduli space $\ove{\mtf{M}}^K(c)$ of log Fano pairs which admits a $\mb{Q}$-Gorenstein degeneration to $(\mb{P}^2,cC_4)$, where $C_4$ is a smooth plane quartic and $c\in(0,\frac{3}{4})$ is a rational number. First of all, ones proves that $\ove{\mtf{M}}^K(\varepsilon)$ is isomorphic to the GIT moduli $\ove{\mtf{M}}^{\GIT}$ of plane quartics. If $C_4$ is a quartic curve such that $(\mb{P}^2,\frac{3}{4}C_4)$ has log canonical singularities, then $(\mb{P}^2,\frac{3}{4}C_4)$ is a K-semistable log Calabi–Yau pair. As $\mb{P}^2$ is K-polystable, then by interpolation of K-stability, the log Fano pair $(\mb{P}^2,cC_4)$ is K-semistable for any $0<c<\frac{3}{4}$.
Therefore, the birational map $\ove{\mtf{M}}^K(c)\dashrightarrow \ove{\mtf{M}}^{\GIT}$ is isomorphic over the open subset $\mtf{M}^{\circ}$ parameterizing quartic curves such that $(\mb{P}^2,\frac{3}{4}C_4)$ has log canonical singularities. From the GIT of quartic curves, one sees that $\ove{\mtf{M}}^{\GIT}\setminus \mtf{M}^{\circ}=[2Q]$, which is the double conic. The K-polystable replacement of $[2Q]$ is locus the pairs $(\mb{P}(1,1,4),c(z^2-f_8(x,y)=0))$.

To some extent, this approach relies on the description of the GIT-(semi/poly)stable objects. Unfortunately, due to the fact that $\Aut(\Sigma_5)=\mtf{S}_5$ is a finite group, every element in the linear series $|-2K_{\Sigma_5}|$ is GIT-stable automatically. On the other hand, the Picard rank $\rho(\Sigma_5)$ of $\Sigma_5$ is five, so potentially there are a huge number of degenerations.

In spite of these two difficulties, by bounding some numerical invariants, we can obtain a classification of the surfaces that can appear in the K-moduli $\ove{M}^K(c)$ for some $0<c<\frac{1}{2}$ (cf. Theorem \ref{4}). Bases on this, our approach is to give for each surface $X$ a stratification of the linear series $|-2K_X|$, for each stratum $Z$, we either prove that $(X,cC)$ is K-unstable for any $C\in Z$, or give a value of $c_i\in(0,\frac{1}{2})$ such that $(X,cC)$ is K-unstable for $0<c<c_i$ but K-semistable for $c=c_i$. This is a \emph{GIT-like} approach to finding all the walls and give the description of the stability conditions for each value $c$. We provide a detailed description of wall crossings for K-moduli spaces $\ove{M}^K(c)$ (cf. Theorem \ref{6}).

\begin{theorem}\label{4}
Let $(X,cD)\in\ove{M}^K(c)$ be a K-polystable pair. Then either $X$ is a quintic del Pezzo surface with at worst du Val singularities, or $X$ is the anticanonical model of the surface of one of the following types:
\begin{enumerate}[(a)]
    \item blow-up of $\mb{P}^2$ at five collinear points;
    \item blow-up of $\mb{P}(1,1,4)$ at four points on the infinity section; 
    \item blow-up of $\mb{F}_2$ at four points, three of which are on the same fiber and the other is on the negative section.
\end{enumerate}
\end{theorem}

\begin{theorem}\label{6} \textup{
(Wall crossings for K-moduli)
} The K-moduli space $\ove{M}^K(c)$ (resp. K-moduli stack $\mtc{M}^K(c)$) is irreducible and normal (resp. smooth) for any $c\in\left(0,\frac{1}{2}\right)$. Moreover, the list of K-moduli walls of $\ove{M}^K(c)$ is $$\left\{
\frac{1}{17},\frac{2}{19},\frac{1}{7},\frac{4}{23},\frac{1}{5},\frac{11}{52},\frac{2}{9},\frac{7}{29},\frac{1}{4},\frac{8}{31},\frac{17}{64},\frac{19}{68}\right\}\bigcup$$ $$\left\{\frac{2}{7},\frac{23}{76},\frac{4}{13},\frac{13}{41},\frac{9}{28},\frac{31}{92},\frac{16}{47},\frac{7}{20},\frac{19}{53},\frac{13}{36},\frac{4}{11},\frac{11}{28}\right\}.$$
    \begin{enumerate}
    \item Among the walls listed above, there are 3 divisorial contractions, and 21 flips. In particular, the Picard number of $\ove{M}^K\left(\frac{1}{2}-\varepsilon\right)$ is $4$.
    \item When $c=c_1=\frac{1}{17}$, the wall crossing morphism $\ove{M}^K(c_1+\varepsilon)\rightarrow\ove{M}^K(c_1-\varepsilon)$ is a divisorial contraction, and the exceptional divisor $\ove{E}_1$ is birational to the sublocus of $\ove{M}_6$ consisting of stable curves of genus six curves with less than five $g^2_6$.
    \item When $c=c_6=\frac{11}{52}$, the wall crossing morphism $\ove{M}^K(c_6+\varepsilon)\rightarrow\ove{M}^K(c_6-\varepsilon)$ is a divisorial contraction. The exceptional divisor $\ove{E}_2$ admits a generically $\mb{P}^1$- fibration to the sublocus of $\ove{M}_6$ consisting of trigonal curves.  
    \item When $c=c_9=\frac{1}{4}$, the wall crossing morphism $\ove{M}^K(c_9+\varepsilon)\rightarrow\ove{M}^K(c_9-\varepsilon)$ is a divisorial contraction. The exceptional divisor $\ove{E}_3$ admits a generically 2-dimensional fibration to the sublocus of $\ove{M}_6$ consisting of plane quintic curves. Moreover, for $\frac{1}{4}<c<\frac{1}{2}$, the $\ove{E}_3(c)$ is isomorphic to the Laza's VGIT $$\ove{M}_{\mb{P}^2}^{\GIT}(t):=\left(|\mtc{O}_{\mb{P}^2}(5)|\times|\mtc{O}_{\mb{P}^2}(1)|\right)\sslash_t \SL(3)$$ of $(1,5)$-plane curves via the relation $t=\frac{25-20c}{1+28c}$.
\end{enumerate}

\end{theorem}

\textbf{The organization of the paper}

In Section 2, we recall some results which are needed in the rest of the paper, including concepts of K-stability and K-moduli spaces, geometry and moduli of genus six curves, and classification of quintic del Pezzo surfaces. In Section 3, we prove that K-moduli spaces $\ove{M}^K(c)$ is isomorphic to the GIT moduli space when $c$ is small, and study the first wall, which is a divisorial contraction, in detail. We will give explicit computation of the $\beta$-invariants for the first wall, as this technique will be used through out the paper.

In Section 4, we concentrate on walls corresponding to ADE del Pezzo surfaces. In Section 5, we study the other two divisorial contractions and show that the exceptional locus is isomorphic to some special VGIT of curves on $\mb{P}^2$ and $\mb{P}^1\times\mb{P}^1$ respectively. It turns out that all these exceptional loci correspond to certain loci of special curves in $\ove{M}_6$. Combining all these together, one proves Theorem \ref{6}.

Finally, in Section 7, we study the birational map from $\ove{M}^K(c)$ to moduli $\ove{M}_6$ of curves and moduli $\mtc{F}^{*}$ of K3 surfaces. In particular, we finish the proof of Theorem \ref{72} and Theorem \ref{7}.
~\\

\textbf{Acknowledgements} The author is greatly indebted to Yuchen Liu for suggesting the problem and providing useful ideas. The author would also like to thank Izzet Coskun, Lawrence Ein for many stimulating discussions. I would also like to express my gratitude to Zhiyuan Li, Zengrui Han, Fei Si and J. Ross Goluboff for insightful conversations and emails.

\section{Preliminaries}\label{119}

\begin{convention}\textup{
Throughout this paper, we work over the field of complex numbers $\mb{C}$. By a curve, we mean a connected projective algebraic curve. For notions and properties of singularities of surface pairs, we refer the reader to \cite[Chapter 2, 4]{KM98}.}
\end{convention}

\subsection{K-stability of log Fano pairs} 

\begin{defn}
Let $X$ be a normal projective variety, and $D$ be an effective $\mb{Q}$-divisor. Then $(X,D)$ is called a \textup{log Fano pair} if $K_X+D$ is $\mb{Q}$-Cartier and $-(K_X+D)$ is ample. A normal projective variety $X$ is called a \textup{$\mb{Q}$-Fano variety} if $(X,0)$ is a klt log Fano pair.
\end{defn}

\begin{defn}
Let $(X,D)$ be an n-dimensional log Fano pair, and $E$ a prime divisor on a normal projective variety $Y$, where $\pi:Y\rightarrow X$ is a birational morphism. Then the \textup{log discrepancy} of $(X,E)$ with respect to $E$ is $$A_{(X,D)}(E):=1+\coeff_{E}(K_Y-\pi^{*}(K_X+D)).$$ We define the \textup{S-invariant} of $(X,D)$ with respect to $E$ to be $$S_{(X,D)}(E):=\frac{1}{(-K_X-D)^n}\int_{0}^{\infty}\vol_Y(\pi^{*}(-K_X-D)-tE)dt,$$ and the \textup{$\beta$-invariant} of $(X,D)$ with respect to $E$ to be $$\beta_{(X,D)}(E):=A_{(X,D)}(E)-S_{(X,D)}(E)$$
\end{defn}

The first definition of K-(poly/semi)stability of log Fano pairs used test configurations. We refer the reader to \cite{Xu21}. There is an equivalent definition using valuations, which is called \emph{valuative criterion} for K-stability. 

\begin{theorem}\label{32} \textup{(cf. \cite{Fuj19,Li17,BX19})} A log Fano pair $(X,D)$ is 
\begin{enumerate}
    \item K-semistable if and only if $\beta_{(X,D)}(E)\geq 0$ for any prime divisor $E$ over $X$;
    \item K-stable if and only if $\beta_{(X,D)}(E)>0$ for any prime divisor $E$ over $X$.
\end{enumerate}

\end{theorem}

The following powerful result is called \emph{interpolation} of K-stability. We only state a version that we will use later. For a more general statement, see for example \cite[Proposition 2.13]{ADL19} or \cite[Lemma 2.6]{Der16}.

\begin{theorem}\label{36}
Let $X$ be a K-semistable $\mb{Q}$-Fano variety, and $D\sim_{\mb{Q}}-rK_X$ be an effective divisor. 
\begin{enumerate}[(1)]
    \item If $(X,\frac{1}{r}D)$ is klt, then $(X,cD)$ is K-stable for any $c\in(0,\frac{1}{r})$;
    \item If $(X,\frac{1}{r}D)$ is log canonical, then $(X,cD)$ is K-semistable for any $c\in(0,\frac{1}{r})$.
\end{enumerate}
\end{theorem}

The following result on openness of K-semistability is useful in proving certain log Fano pairs is K-semistable.

\begin{theorem}\label{35} \textup{(cf. \cite{BLX19,Xu20})}
Let $(\mtf{X},\mtf{D})\rightarrow B$ be a $\mb{Q}$-Gorenstein family of log Fano pairs over a normal base $B$. Then
$$\{b\in B:(\mtf{X}_{\ove{b}},\mtf{D}_{\ove{b}}) \textup{ is K-semistable}\}$$
is a Zariski open subset of $B$.
\end{theorem}

Recall that the volume of a divisor $D$ on an n-dimensional normal projective variety $Y$ is $$\vol_Y(D):=\lim_{m\to \infty}\frac{\dim H^0(Y,mD)}{m^n/n!}.$$ The divisor $D$ is \emph{big} by definition if and only if $\vol_Y(D)>0$.

\begin{defn}
Let $x\in (X,D)$ be an n-dimensional klt singularity. Let $\pi:Y\rightarrow X$ be a birational morphism such that $E\subseteq Y$ is an exceptional divisor whose center on $X$ is $\{x\}$. Then the \textup{volume} of $(x\in X)$ with respect to $E$ is $$\vol_{x,X,D}(E):=\lim_{m\to \infty}\frac{\dim\mtc{O}_{X,x}/\{f:\ord_E(f)\geq m\}}{m^n/n!},$$ and the \textup{normalized volume} of $(x\in X)$ with respect to $E$ is $$\wh{\vol}_{x,X,D}(D):=A_{(X,D)}(D)^n\cdot\vol_{x,X,D}(E).$$ We define the \textup{local volume} of $x\in(X,D)$ to be $$\wh{\vol}(x,X,D):=\inf_{E}\wh{\vol}_{x,X,D}(D),$$ where $E$ runs through all the prime divisor over $X$ whose center on $X$ is $\{x\}$.
\end{defn}

\begin{theorem}\label{34} \textup{(cf. \cite{Fuj18,LL19,Liu18})}
Let $(X,D)$ be an n-dimensional K-semistable log Fano pair. Then for any $x\in X$, we have $$(-K_X-D)^n\leq \left(1+\frac{1}{n}\right)^n\widehat{\vol}(x,X,D).$$
\end{theorem}

Another result which we will use through out this paper is the equivariant K-stability. The following result is essentially due to \cite{Zhu21}. 

\begin{defn}
Let $(X,D)$ be a log Fano surface pair with klt singularities. If the maximal torus of $\Aut^0(X,D)$ is $\mb{G}_m$, then we call $(X,D)$ a \textup{complexity one $\mb{T}$-pair}. A $\mb{G}_m$-equivariant divisor $F$ on $X$ is called \textup{vertical} if a maximal $\mb{G}_m$-orbit in $F$ has dimension $1$; is called \textup{horizontal} if otherwise.      
\end{defn}

\begin{theorem}\label{71} \textup{(cf. \cite[Theorem 1.31]{CA23})}
Let $(X,D)$ be a log Fano surface pair of complexity one, with the $\mb{G}_m$-action $\lambda$. Then $(X,D)$ is K-polystable if and only if the following conditions hold:
\begin{enumerate}[(1)]
    \item $\beta_{(X,D)}(F)>0$ for each vertical $\mb{G}_m$-equivariant prime divisor on $X$;
    \item $\beta_{(X,D)}(F)=0$ for each horizontal $\mb{G}_m$-equivariant prime divisor on $X$;
    \item $\Fut_{(X,D)}(v)=0$ for the valuation $v$ induced by $\lambda$. 
\end{enumerate}
\end{theorem}

\subsection{K-moduli of log Fano pairs}

In this subsection, we briefly review results on the K-moduli theory for log Fano pairs.

\begin{defn}
Let $\pi:(\mtc{X},\mtc{D})\rightarrow S$ be a proper flat morphism to a reduced scheme with normal, geometrically connected fibers of pure dimension $n$, where $\mtc{D}$ is an effective $\mb{Q}$-divisor on $\mtc{X}$ which does not contain any fiber of $\pi$. Then $\pi$ is called a \textup{$\mb{Q}$-Gorenstein flat family of log Fano pairs} if $-(K_{\mtc{X}/B}+\mtc{D})$ is $\mb{Q}$-Cartier and ample over $S$.
\end{defn}

\begin{defn}
Let $0<c<1/r$ be a rational number and $(X,cD)$ be a log Fano surface pair such that $D\sim -rK_X$ for some fixed $r\in\mb{Q}_{\geq 1}$. A $\mb{Q}$-Gorenstein flat family of log Fano pairs $\pi:(\mtc{X},c\mtc{D})\rightarrow C$ over a pointed smooth curve $(0\in C)$ is called a \textup{$\mb{Q}$-Gorenstein smoothing} of $(X,D)$ if  
\begin{enumerate}[(1)]
    \item the divisors $\mtc{D}$ and $K_{\mtc{X}/C}$ are both $\mb{Q}$-Cartier, $f$-ample, and $\mtc{D}\sim_{\mb{Q},\pi}-rK_{\mtc{\mtc{X}}/C}$;
    \item both $\pi$ and $\pi|_{\mtc{D}}$ are smooth over $C\setminus\{0\}$, and
    \item $(\mtc{X}_0,c\mtc{D}_0)\simeq (X,cD)$.
\end{enumerate}
\end{defn}

The following classical result classifies all the klt surface singularities that admit a $\mb{Q}$-Gorenstein smoothing.

\begin{theorem}\label{37} \textup{(cf. \cite[Proposition 3.10]{KSB88})} A klt surface singularity $X_0$ which admits a $\mb{Q}$-Gorenstein smoothing is either an ADE singularity or a cyclic quotient singularity of the form $\mb{A}^2/\langle\xi\rangle$, where $\xi$ is a primitive $dn^2$-th root of unity acting as $$\xi(u,v)=(\xi u,\xi^{dna-1} v),$$ where $0<a<n$ is coprime to $n$. In the later case, we say that $X_0$ is a \textup{$ \frac{1}{dn^2}(1,dna-1)$ singularity}.

\end{theorem}

Now we introduce the CM line bundle of a flat family of log Fano pairs, which is a functorial line bundle over the base (cf. \cite{PT06,PT09,Tia97}). 

Let $\pi:\mtc{X}\rightarrow S$ be a proper flat morphism of connected schemes with $S_2$ fibers of pure dimension $n$, and $\mtc{L}$ be an $\pi$-ample line bundle on X. Let $\mtc{D}:=\sum a_i\mtc{D}_i$ be a relative Mumford $\mb{Q}$-divisor on $\mtc{X}$ over $S$, where each component $\mtc{D}_i$ is a relative Mumford divisor and $a_i\in(0,1)$ is a rational number. Since we consider only $\mb{Q}$-Gorenstein smoothable family of log Fano pairs over reduced base schemes, then we may assume also that each divisor $\mtc{D}_i$ is flat over $S$ (cf. \cite[Proposition 2.12]{ADL21}). By \cite{KM76}, there are line bundles $\lambda_i(\mtc{X},\mtc{L})$ on $S$ such that $$\det(\pi_{!}(\mtc{L}^k))=\lambda_{n+1}^{\otimes\binom{k}{n+1}}\otimes \lambda_n^{\otimes\binom{k}{n}}\otimes \cdots\otimes\lambda_0$$ for any $k\gg0$. Write the Hilbert polynomial for each fiber $\mtc{X}_s$ as $$\chi(\mtc{X}_s,\mtc{L}^k_{s})=b_0k^n+b_1k^{n-1}+O(k^{n-2}).$$
 
\begin{defn}
    The \textup{CM line bundle} and the \textup{Chow line bundle} of the polarized family $(\pi:\mtc{X}\rightarrow S,\mtc{L})$ are $$\lambda_{CM,\pi,\mtc{L}}:=\lambda_{n+1}^{n(n+1)+\frac{2b_1}{b_0}}\otimes \lambda_n^{-2(n+1)},\quad \lambda_{\Chow,\pi,\mtc{L}}:=\lambda_{n+1}.$$ The \textup{log CM $\mb{Q}$-line bundle} of the family $(\pi:\mtc{X}\rightarrow S,\mtc{L},\mtc{D})$ is $$\lambda_{\CM,\pi,\mtc{D},\mtc{L}}:=\lambda_{\CM,\pi,\mtc{L}}-\frac{n(\mtc{L}_s^{n-1}.\mtc{D}_s)}{(\mtc{L}^n_s)}\lambda_{\Chow,\pi,\mtc{L}}+(n+1)\lambda_{\Chow,\pi|_{\mtc{D}},\mtc{L}|_{\mtc{D}}},$$ where we denote $(\mtc{L}_s^{n-1}.\mtc{D}_s):=\sum a_i(\mtc{L}_s^{n-1},\mtc{D}_{i,s})$, and $\lambda_{\Chow,\pi|_{\mtc{D}},\mtc{L}|_{\mtc{D}}}:=\bigotimes \lambda^{\otimes a_i}_{\Chow,\pi|_{\mtc{D}_i},\mtc{L}|_{\mtc{D}_i}}$.
\end{defn}

When restricting to $\mb{Q}$-Gorenstein smoothable families of log Fano pairs $\pi:(\mtc{X},c\mtc{D})\rightarrow S$ over a reduced base scheme, where $\mtc{D}\sim_{\mb{Q},\pi}-rK_{\mtc{X}}$ for some $r\in \mb{Q}_{>0}$, we can also make the following definition.

\begin{defn}
    The \textup{Hodge $\mb{Q}$-line bundle $\lambda_{\Hodge,\pi,r^{-1}\mtc{D}}$} is the $\mb{Q}$-linear equivalence class of $\mb{Q}$-Cartier $\mb{Q}$-divisors on $S$ such that $$K_{\mtc{X}/S}+r^{-1}\mtc{D}\sim_{\mb{Q}}\pi^{*}\lambda_{\Hodge,\pi,r^{-1}\mtc{D}}.$$
\end{defn}

The following result is called \emph{K-moduli Theorem}, which is contributed by an amount of people (cf. \cite{ABHLX20,BLX19,BX19,BHLLX21,CP21,Jia20,LWX21,LXZ22,XZ20}).

\begin{theorem} \textup{(K-moduli Theorem)}
Let $c\in\left(0,\frac{1}{r}\right)$ be a rational number, and $\chi$ be the Hilbert polynomial of the anti-canonically polarized smooth quintic del Pezzo surface. Consider the moduli pseudo-functor $\mtc{M}^K(c)$ sending a reduced base $S$ to

\[
\left\{(\mtc{X},\mtc{D})/S\left| \begin{array}{l}(\mtc{X},c\mtc{D})/S\textrm{ is a $\mb{Q}$-Gorenstein smoothable log Fano family,}\\ \mtc{D}\sim_{S,\mb{Q}}-rK_{\mtc{X}/S},~\textrm{each fiber $(\mtc{X}_s,c\mtc{D}_s)$ is K-semistable,}\\ \textrm{and $\chi(\mtc{X}_s,\mtc{O}_{\mtc{X}_s}(-mK_{\mtc{X}_s}))=\chi(m)$ for $m$ sufficiently divisible.}\end{array}\right.\right\}.
\]
Then there is a reduced Artin stack, still denoted by $\mtc{M}^K_{c}$, of finite type over $\mb{C}$ which represents the pseudo-functor. The $\mb{C}$-points of $\mtc{M}^K(c)$ parameterize K-semistable $\mb{Q}$-Gorenstein smoothable log Fano pairs $(X,cD)$ with Hilbert polynomial $\chi(X,\mtc{O}_X(-mK_X))=\chi(m)$ for sufficiently divisible $m\gg 0$ and $D\sim_{\mb{Q}}-rK_X$. Moreover, the Artin stack $\mtc{M}^K(c)$ admits a good moduli space $\ove{M}^K(c)$, which is a normal projective reduced scheme of finite type over $\mb{C}$, whose $\mb{C}$-points parameterize K-polystable log Fano pairs. The CM $\mb{Q}$-line bundle $\lambda_{\CM,c}$ on $\mtc{M}^K(c)$ descends to an ample $\mb{Q}$-line bundle $\Lambda_{\CM,c}$ on $\ove{M}^K(c)$.
\end{theorem}

When we vary the coefficient $c$ in the log Fano region, the K-moduli has the following wall-crossing structure, locally presented by VGIT.

\begin{theorem} \textup{(cf. \cite[Theorem 1.2]{ADL19})} There are rational numbers $$0=c_0<c_1<c_2<\cdots<c_n=\frac{1}{r}$$ such that for every $0\leq j<n$, the K-moduli stacks $\mtc{M}^K(c)$ are independent of the choice of $c\in(c_j,c_{j+1})$. Moreover, for every $0< j<n$ and $0<\varepsilon\ll1$, one has open immersions $$\mtc{M}^K(c_j-\varepsilon)\hookrightarrow \mtc{M}^K(c_j)\hookleftarrow \mtc{M}^K(c_j+\varepsilon),$$ which descend to projective birational morphisms $$\ove{M}^K(c_j-\varepsilon)\rightarrow \ove{M}^K(c_j)\leftarrow \ove{M}^K(c_j+\varepsilon).$$

\end{theorem}

\subsection{Geometry of smooth algebraic curves of genus six}

In this part, we introduce some geometry of smooth projective curves of genus six and their K-moduli space. 

By Brill-Noether theorem (cf. \cite[Chapter V]{ACGH}), we know that a general curve $C$ of genus six has finitely many linear series of degree $6$ and dimension $2$, denoted by $g^2_6$. Using one of them, one can map this curve birationally to its image in $\mb{P}^2$ such that the image is a plane sextic curve with four nodal singularities. Blowing up $\mb{P}^2$ at these four points, we obtain a smooth del Pezzo surface of degree $5$, which is unique up to isomorphism. We denote the surface by $\Sigma_5$. The canonical embedding $C\hookrightarrow \Sigma_5$ realizes $C$ as a divisor on $\Sigma_5$ of the class $-2K_{\Sigma_5}$. In fact, a general genus six curve has exactly five $g^2_6$, and each curve with exactly five $g^2_6$ can be canonically embedded into $\Sigma_5$. There are also curves with less than five $g^2_6$, and these curves can be embedded into a unique quintic del Pezzo surface with at worst rational double points (cf. \cite[Claim 5.14]{AH81}).

A smooth genus six curve is called \emph{special} if it has infinitely many $g^2_6$. A special genus $6$ curve is of one of the following types:
\begin{enumerate}[(a)]
    \item smooth plane quintic curves,
    \item hyperelliptic curves,
    \item trigonal curves, or
    \item bielliptic curves.
\end{enumerate}

We will use the following property of trigonal curves later.

\begin{prop}
A smooth trigonal genus $6$ curve $C$ is canonically embedded 
\begin{enumerate}[(a)]
    \item either in $\mb{F}_0=\mb{P}^1\times\mb{P}^1$ of the class $\mtc{O}_{\mb{P}^1\times\mb{P}^1}(3,4)$;
    \item or in $\mb{F}_2$ of the class $3E+7F$, where $E,F$ are the divisor class of the negative section and a fiber.
\end{enumerate}
\end{prop}

\begin{proof}
    It is well-known that the canonical model of a trigonal curve $C$ lies on a unique Hirzebruch surface $\mb{F}_e$, which is obtained by taking the union of the lines spanned each three points in the $g^1_3$. It is easy to check that the curves in the two listed linear series are of genus $6$, and a general one is smooth.
    
    Now suppose that the class of $C$ as a divisor on $\mb{F}_e$ is $3E+kF$, where $E,F$ are the class of the negative section and a fiber respectively. As $C$ is irreducible, then $k\geq 2e$. In our case, the genus of $C$ is six, so by adjunction formula we have that $$10=2g-2=(C.C+K_{\mb{F}_e})=(3E+kF.E+(k-2-e)F)=4k-6e-6\geq 6e-6.$$ It follows that $e\leq 2$. Notice also that $e$ has to be an even number. Therefore we have that either $e=0$ or $e=2$, and $k=4$ and $k=7$ correspondingly.  
\end{proof}

Combining the following results, we see that the moduli space $M_6$ of smooth genus 6 curves admits a stratification $$M_6=U\sqcup Z_{<5}\sqcup Z_q\sqcup Z_h\sqcup Z_b\sqcup Z_t,$$ where $U$ and $Z_{<5}$ are non-special curves with exactly five and less than five $g^2_6$ respectively, and $Z_q,Z_h,Z_b,Z_t$ are the locus of quintic, hyperelliptic, bielliptic, and trigonal curves, respectively. By counting dimension, we see that the closure of $Z_{<5}$ is a divisor in $M_6$, which generates $\Pic(M_6)_{\mb{Q}}$. Taking the singular DM-stable curves into consideration, we also have that $$\ove{M}_6\setminus M_6=\Delta_0\cup  \Delta_1\cup\Delta_2\cup\Delta_3,$$ where $\Delta_0$ is the closure of the locus of irreducible nodal curves and $\Delta_i$ is the closure of the locus of reducible curves with one component of genus $i$.

\subsection{Classification of quintic del Pezzo surfaces}\label{102}

In this section, we list the classification of quintic del Pezzo surfaces of Gorenstein index $\leq 3$.

The following result enables us to classify all Gorenstein quintic del Pezzo surfaces, that is, those surfaces with at worst rational double points (or ADE singularities).

\begin{theorem}\textup{(cf. \cite[Theorem 3.4]{HW81})}
Let $X$ be a normal projective Gorenstein del Pezzo surface of degree $5$. Then $X$ is the anticanonical-ample model of blow-ups of $\mb{P}^2$ along four points in almost general position.
\end{theorem}

According to this, we can give a more detailed classification of them: a normal projective Gorenstein del Pezzo surface of degree $5$ is isomorphic to one of the following surfaces.
\begin{enumerate}[(i)]
    \item $X=\Sigma_5$ is the blow-up of $\mb{P}^2$ along four general points. In this case, $X$ is smooth;
    \item $X=X_1$ is obtained by blowing up $\mb{P}^2$ at three general points $p_1,p_2,p_3$, then at a point $p_4$ on the exceptional divisor over $p_3$ and contracting the (-2)-curve. In this case, $X$ has an $A_1$-singularity.
    \item $X=X_{1,1}$ is obtained by blowing up $\mb{P}^2$ at two general pairs of infinitely close points $(p_1,p_2)$ and $(p_3,p_4)$, then contracting the two (-2)-curves. In this case, $X$ has two $A_1$-singularities.
    \item $X=X_{2}$ is obtained by blowing up $\mb{P}^2$ at three distinct collinear points $p_1,p_2,p_3$, and at a general point on the exceptional divisor over $p_3$, then contracting the two (-2)-curves. In this case, $X$ has an $A_2$-singularity.
    \item $X=X_{1,2}$ is obtained by blowing up $\mb{P}^2$ at two distinct points $p_1,p_2$, and at a general point on the exceptional divisor over $p_1$ and the point on the exceptional divisor over $p_2$ which corresponds to the line spanned by $p_1,p_2$, then contracting the three (-2)-curves. In this case, $X$ has an $A_1$-singularity and an $A_2$-singularity.
    \item $X=X_{3}$ is obtained by blowing up $\mb{P}^2$ at two distinct points $p_1,p_2$, and at the point $p_3$ on the exceptional divisor over $p_2$ which corresponds to the line spanned by $p_1,p_2$ and a general point on the exceptional divisor over $p_3$, then contracting the three (-2)-curves. In this case, $X$ has an $A_3$-singularity.
    \item $X=X_{4}$ is obtained by blowing up $\mb{P}^2$ at a point $p_1$, a general point $p_2$ on the exceptional divisor over $p_1$, a general point $p_3$ on the exceptional divisor over $p_2$, and a general point $p_4$ on the exceptional divisor over $p_3$, then contracting the four (-2)-curves. In this case, $X$ has an $A_4$-singularity.
\end{enumerate}

\begin{remark}
\textup{Here the index denotes the singularities appearing on the surface. For example, the surface $X_{1,2}$ means the surface has an $A_1$-singularity and an $A_2$-singularity.}
\end{remark}

In \cite{Nak07}, Nakayama gives a list of del Pezzo surfaces of Gorenstein index $2$ with $(-K_X)^2=5$. As the notations in the original paper is complicated, we give a more explicit description here.

\begin{theorem}\label{40} \textup{(cf. \cite[Table 6]{Nak07})}
Each del Pezzo surface $X$ of Cartier index $2$ with $(-K_X)^2=5$ is the canonical model of one of the following surfaces:
\textup{\begin{enumerate}[(i)]
    \item $[1]_0$: blow-up of $\mb{P}^2$ along $5$ collinear points;
    \item $[0;1,0]_0$: blow-up of $\mb{F}_0$ along $4$ collinear points on a fixed section;
    \item $[2;1,1]_{+}(0,0)$: blow-up of $\mb{F}_2$ along $3$ collinear points on a fiber $l$ outside the negative section $\sigma$ and a point on the negative section;
    \item $[2;1,1]_{+}(1,1)$: blow-up of $\mb{F}_2$ along $2$ points on a fiber $l$ outside the negative section $\sigma$ and a length $2$ subscheme $\Delta$ supported at $\sigma\cap l$ such that $\length(\Delta\cap l)=1$ and $\length(\Delta\cap \sigma)=1$;
    \item $[2;1,1]_{+}(1,2)$: blow-up of $\mb{F}_2$ at one point on a fiber $l$ outside the negative section $\sigma$ and a length $3$ subscheme $\Delta$ supported at $\sigma\cap l$ such that $\length(\Delta\cap l)=2$ and $\length(\Delta\cap \sigma)=1$;
    \item $[2;1,1]_{+}(1,3)$: blow-up of $\mb{F}_2$ at a length $4$ subscheme $\Delta$ supported at $\sigma\cap l$ such that $\length(\Delta\cap l)=3$ and $\length(\Delta\cap \sigma)=1$.
\end{enumerate}}
\end{theorem}

There is also a classification of del Pezzo surfaces of Gorenstein index $3$ (cf. \cite[Table 7]{FY17}). If $(X,cD)$ is a del Pezzo pair of Cartier index $3$ appearing in some K-moduli space $\ove{M}^K(c)$, then the singularities at which $X$ has Cartier index $3$ must lie on the support of $D$. Combining \cite[Table 7]{FY17} and Lemma \ref{31}, the only index $3$ surfaces that may appear in $\ove{M}^K(c)$ are of type $[2,1;2]_{1G}$ or $[2,2;1]_{1A}$, i.e. the surfaces obtained from the blow-up of $\mb{F}_2$ along 5 points on a fiber, or the the blow-up of $\mb{F}_2$ along 2 points on a fiber and $3$ points on the negative section. Notice that these two types of surfaces have $\frac{1}{9}(1,2)$ singularity.

However, we will show later (cf. Proposition \ref{33}) that such surfaces in fact do not appear in the K-moduli space $\ove{M}^K(c)$ for any $0<c<1/2$. As a consequence, any Gorenstein index $3$ surface in the class $[2,1;2]_{1G}$ does not appear in the K-moduli by the openness of the K-stability.

\section{Relation to the GIT moduli space and the first wall}

In this section, we study the K-moduli space $\ove{M}^K(c)$ for $0<c\ll1$. In particular, we prove Theorem \ref{72} (1). Moreover, we will analyze the first wall $c_1=\frac{1}{17}$ and provide all computational details for the further use.

\subsection{Index bound and restrictions on singularities}

We will need the following result, which gives a restriction on the surfaces that can appear in the moduli space.

\begin{prop}\label{31}
Let $(X,cD)$ be a K-semistable log Fano pair which admits a $\mb{Q}$-Gorenstein smoothing to $(\Sigma_5,cC)$, where $C\in |-2K_{\Sigma_5}|$ and $0<c<1/2$. Suppose that $x\in X$ is a point such that the Gorenstein index $\Ind(x,K_X)\geq3$. Then we have $\Ind(x,K_X)=3$ and locally it is a $\frac{1}{9}(1,2)$-singularity.
\end{prop}

\begin{proof}
Let $n:=\Ind(x,K_X)\geq3$ be the Gorenstein index at $x$. If $x\notin D$, then the assumption $D+2K_X\sim 0$ implies that $n\leq 2$. Thus we must have that $x\in D$. As $x\in X$ is a $\mb{Q}$-Gorenstein smoothable singularity of index $n\geq3$, then it is a cyclic quotient singularity of type $\frac{1}{dn^2}(1,dna-1)$, where $\gcd(a,n)=1$ and $0<a<n$ (cf. \cite[Proposition 3.10]{KSB88}). Since $(X,cD)$ is K-semistable, then it follows from Theorem \ref{34} that $$5(1-2c)^2=(-K_{X}-cD)^2\leq \frac{9}{4}\widehat{\vol}(x,X,cD).$$  
Let $(\tilde{x},\tilde{X},\tilde{D})\rightarrow(x,X,D)$ be a smooth covering at $x$ of degree $dn^2$. By \cite[Theorem 1.3]{XZ21}, one has that
\begin{equation}\label{38}
    dn^2\cdot\frac{4}{9}\cdot 5(1-2c)^2\leq dn^2\widehat{\vol}(x,X,cD)=\widehat{\vol}(\tilde{x},\tilde{X},c\tilde{D})\leq(2-c\ord_{\tilde{x}}\tilde{D})^2<4.
\end{equation} 
If $\ord_{\tilde{x}}\tilde{D}\geq 4$, then we have $$dn^2\cdot \frac{4}{9}5(1-2c)^2\leq 4(1-2c)^2,$$ and thus $dn^2\leq 9/5$. This forces $(d,n)=(1,1)$, which is a contradiction.

Now we have that $\ord_{\tilde{x}}\tilde{D}\leq 3$. Let $u,v$ be local coordinates near $\tilde{x}\in\tilde{X}$ such that the cyclic group action is scaling on each coordinate, and $u^iv^j$ be a monomial appearing in the local equation of $\tilde{D}$ such that $i+j=\ord_{\tilde{x}}\tilde{D}$. The relation $D+2K_X\sim 0$ implies that $$i+(dna-1)j\equiv 2dna \mod dn^2.$$ In particular, we have $i\equiv j \mod dn$. If $i=j$, then both of them are equal to $1$, and $dn^2|dna$, which implies that $n=1$ since $\gcd(n,a)=1$. If $i\neq j$, then either $dn|2$ or $dn|3$. Since $n\geq 3$, we must have $(d,n)=(1,3)$, and hence $x\in X$ is a $\frac{1}{9}(1,2)$-singularity (notice that $\frac{1}{9}(1,2)$-singularity and $\frac{1}{9}(1,5)$-singularity are the same).   

\end{proof}

\subsection{The first chamber: $\ove{M}^K(\varepsilon)$}

Let ${\bf{P}}=|-2K_{\Sigma_5}|$ be the complete linear series. Recall that we define the GIT moduli stack to be $\mtc{M}^{\GIT}:=\left[{\bf{P}}/\Aut(\Sigma_5)\right]$ and its good moduli space $\ove{M}^{\GIT}:={\bf{P}}/\Aut(\Sigma_5)$. In fact, as $\Aut(\Sigma_5)\simeq \mtf{S}_5$ is finite, this is a Deligne-Mumford stack with its coarse moduli space.

\begin{theorem}\label{55}
Let $0<\varepsilon\ll 1$ be a rational number. Then there is an isomorphism of Artin stacks $\mtc{M}^{K}(\varepsilon)\simeq {\mtc{M}}^{\GIT}$, which descends to an isomorphism $\ove{M}^{K}(\varepsilon)\simeq \ove{M}^{\GIT}$ of good moduli spaces. 
\end{theorem}

\begin{proof}

By \cite[Corollary 1.8]{PW18}, the smooth quintic del Pezzo surface $\Sigma_5$ is K-stable. Since being K-stable is equivalent to being uniform K-stable (cf. \cite[Theorem 1.6]{LXZ22}), which is an open condition in the sense of \cite{Fuj19b}, then for any divisor $D\in|-2K_{\Sigma_5}|$, we have that $(X,\varepsilon D)$ is K-stable for $0<\varepsilon\ll1$. 

We now show that if $(X,\varepsilon D)$ is K-semistable, then we have $X$ is smooth, and hence $X\simeq \Sigma_5$ and $D\in|-2K_{\Sigma_5}|$. For any point $x\in X$, by \cite[Proposition 4.6]{LL19}, one has that $$5(1-2\varepsilon)^2=(-K_X-\varepsilon D)^2\leq \left(\frac{3}{2}\right)^2\widehat{\vol}(x,X,D)\leq\frac{9}{4}\widehat{\vol}(x,X).$$ If $x$ is a singular point on $X$, then locally we have $(x\in X)\simeq(\mb{C}^2/G)$ for some nontrivial finite group $G$. Then it follows from in \cite[Theorem 2.7]{LX19} that $$|G|\leq 4\cdot\frac{9}{4}\cdot\frac{1}{5(1-2\varepsilon)^2}<2$$ for any $0<\varepsilon\ll1$. Thus $G$ is the trivial group and $X$ is smooth as desired.

Now let $\pi:(\Sigma_5\times \bf{P},\varepsilon \mtc{D})\rightarrow \bf{P}$ be the natural family of K-stable log Fano pairs over $\bf{P}$, which is $\mtf{S}_5$-equivariant. Notice that the K-stability and GIT-stability is trivially equivalent as each fiber (resp. point on the base) is K-stable (resp. GIT-stable). Therefore, by the universality of K-moduli and GIT-moduli stacks, there exists a natural isomorphism $\psi:\mtc{M}^{\GIT}\rightarrow \mtc{M}^K(\varepsilon)$ between Artin stacks, which descends to an isomorphism $\ove{M}^{K}(\varepsilon)\simeq \ove{M}^{\GIT}$ of good moduli spaces.

\end{proof}

\begin{remark}\textup{
Notice that the Picard number of $\ove{M}^K({\varepsilon})$ is one since it is a finite quotient of the projective space $\bf{P}$. The locus consisting of $[(\Sigma_5,\varepsilon C)]$ with $C$ a singular curve is a divisor $Z_0$ in $\ove{M}^K(\varepsilon)$. The natural birational map $\ove{M}^K(\varepsilon)\dashrightarrow \ove{M}_6$ will send $Z_0$ to $\Delta_0$, the locus of irreducible nodal curves.
}\end{remark}

We will show in the following section that the first wall is $c_1=1/17$ (cf. Proposition \ref{5}), which will imply the following.

\begin{corollary}
For any rational number $o<c<1/17$, there is an isomorphism of Artin stacks $\mtc{M}^{K}(c)\simeq {\mtc{M}}^{\GIT}$, which descends to an isomorphism $\ove{M}^{K}(c)\simeq \ove{M}^{\GIT}$ of good moduli spaces. 
\end{corollary}

\subsection{The first wall $c_1=\frac{1}{17}$: curves with four $g^2_6$}\*

Let $\Sigma_5$ be the smooth del Pezzo surface obtained by blowing up $\mb{P}^2$ along $4$ general points $p_1,...,p_4$ with exceptional divisors $E_1,...,E_4$ respectively. Let $L_1$ and $L_2$ be the proper transforms of the lines connecting $p_1,p_2$ and $p_3,p_4$ respectively. Then $D_{\frac{1}{17}}:=4L_1+2L_2+2E_1+2E_2$ is a divisor in $|-2K_{\Sigma_5}|$.

\begin{theorem}\label{1}
The log Fano pair $(\Sigma_5,cD_1)$ is K-polystable if and only if $0<c<\frac{1}{17}$. 
\end{theorem}

\begin{lemma}\label{2}
If the pair $(\Sigma_5,cD_{\frac{1}{17}})$ is K-semistable, then $0<c\leq\frac{1}{17}$.
\end{lemma}

\begin{proof}
Let us compute the $\beta$-invariant of $(\Sigma_5,cD_{\frac{1}{17}})$ with respect to the divisor $L_1$. First we have that $A_{(\Sigma_5,cD_{\frac{1}{17}})}(L_1)=1-4c$ by definition. To compute $S$-invariant, we need to do Zariski decomposition for the divisor $-K_{\Sigma_5}-tL_1$. Write $-K_{\Sigma_5}-tL_1=P(t)+N(t)$, where $P(t)$ and $N(t)$ are the positive and negative parts respectively in the Zariski decomposition (cf. \cite[Section 2.3.E]{Laz17}). Then we have that 

\begin{equation}
    P(t)= \begin{cases}
-K_{\Sigma_5}-tL_1,\quad &0\leq t\leq 1 \\
-K_{\Sigma_5}-tL_1-(t-1)(L_2+E_1+E_2),\quad & 1\leq t\leq 2,
\end{cases}
\end{equation}

and that

\begin{equation}
    \vol(-K_{\Sigma_5}-tL_1)=P(t)^2= \begin{cases}
5-2t-t^2,\quad &0\leq t\leq 1 \\
2(2-t)^2,\quad & 1\leq t\leq 2,
\end{cases}
\end{equation}

It follows that
\begin{equation}
\begin{split}
    S_{(\Sigma_5,cD_{\frac{1}{17}})}(L_1)&=\frac{1-2c}{5}\int_0^{2}\vol(-K_{\Sigma_5}-tL_1)dt\\
    &=\frac{1-2c}{5}\left(\int_0^{1}5-2t-t^2dt+\int_1^{2}2(2-t)^2dt\right)=\frac{13(1-2c)}{15}.
\end{split}
\end{equation} Thus $\beta_{(\Sigma_5,cD_{\frac{1}{17}})}(L_1)\geq0$ imposes the condition $c\leq 1/17$.
\end{proof}

Let $X_n$ be a degree $5$ nodal del Pezzo surface with one $A_1$-singularity. Here we use the subscript $n$ to denote \emph{nodal}. In fact, the surface $X_n$ is exactly the $X_1$ in the classification of ADE quintic del Pezzo in Section \ref{102}. We can obtain $X_n$ by blowing up $\mb{P}^2$ at four points $q_1,...,q_4$, where $q_3$ and $q_4$ are infinitely near, and then contracting the $(-2)$-curve. Let $D_n=4l_1+2l_2+2F_1+2F_2$ be a divisor in $|-2K_{X_n}|$, where the $F_i$'s are the exceptional divisors over $q_i$, and $l_1$ and $l_2$ are the proper transforms of the lines $\ove{q_1q_2}$ and $\ove{q_3q_4}$ respectively.  

\begin{lemma}\label{104}
The log Fano pair $(X_n,cD_n)$ is K-semistable if and only if $c=\frac{1}{17}$. Moreover, the pair $(X_n,\frac{1}{17}D_n)$ is K-polystable.
\end{lemma}

\begin{proof}
The pair $(X_n,cD_n)$ is a complexity one $\mb{T}$-variety with the $\mb{G}_m$-action $\lambda$, which induced from the $\mb{G}_m$-actions on $\mb{P}^2$ given by scaling between $q_3$ and the line $l_1$. Then the only horizontal divisor on $X_n$ is $l_1$ and all vertical divisors on $X_q$ are $l_2$, $F_1,F_2,F_4$ and the lines joining $q_3$ and $l_1$, which are of class $H$.

Now we apply Theorem \ref{71} to check that $(X_n,\frac{1}{17}D_n)$ is K-polystable if and only if $c=1/17$. It suffices to compute the $\beta$-invariants with respect to any $\mb{G}_m$-invariant divisors on $X_n$. It is easy to check that when $c=1/17$, the $\beta$-invariants of the vertical divisors are positive and $$\beta_{(X_n.cD_n)}(l_1)=\frac{2(1-17c)}{15}=0$$ by the same computation in the proof of Lemma \ref{2}. In fact, we also have $$\beta_{(X_n.cD_n)}(E)=1-\frac{17(1-2c)}{15}=\frac{2(17c-1)}{15},$$ which is non-negative if and only if $c\geq 1/17$, where $E$ is the exceptional divisor over the singularity. Observe that there is a special test configuration $(\mtf{X},c\mtf{D})\rightarrow \mb{A}^1$ of $(X_n,cD_n)$ induced by the $\mb{G}_m$-action $\lambda$ such that $$\Fut_{(X_n,cD_n)}(\lambda)=\Fut(\mtf{X},c\mtf{D};-K_{\mtf{X}/\mb{A}^1}-c\mtf{D}).$$ Then one can apply \cite[Lemma 3.3, 3.4]{Xu21} to conclude that the equality $\Fut_{(X_n,cD_n)}(\lambda)=0$ is equivalent to $\beta_{(X_n.cD_n)}(l_1)=0$.
\end{proof}

\begin{defn}
    We say that a log Fano pair $(X,D)$ \emph{specially degenerates} to another log Fano pair $(X_0,D_0)$ if there is a fibration $(\mtf{X},\mtf{D})\rightarrow \mb{A}^1$ such that $(\mtf{X}_t,\mtf{D}_t)\simeq (X,D)$ for any $t\neq 0$ and $(\mtf{X}_0,\mtf{D}_0)\simeq (X,D)$.
\end{defn}

\begin{lemma}
The log Fano pair $(\Sigma_5,cD_1)$ specially degenerates to $(X_n,cD_n)$.
\end{lemma}

\begin{proof}
Take a trivial fibration $\mtf{X}':=\Sigma \times \mb{A}^1\rightarrow \mb{A}^1$. Let $\mtf{X}\rightarrow \mtf{X}'$ be the blow-up of $\mtf{X}'$ along $l_1\times\{0\}$, where $l_1$ comes from the line in $\mb{P}^2$ passing through $p_1,p_2$. Then $\pi:\mtf{X}\rightarrow \mb{A}^1$ is also a fibration over $\mb{A}^1$, where the general fibers are all $\mb{P}^2$, and the central fiber is $\hat{X}_1\cup \hat{X}_2$. Here $\hat{X}_1\simeq\Sigma_5$ is the proper transform of $\Sigma_5\times\{0\}$, and $\hat{X}_2$ is isomorphic to the Hirzebruch surface $\mb{F}_1$. Now we run $\hat{X}_1$-MMP. One contracts three curves $E_1\times\{0\}$, $E_2\times\{0\}$ and $l_2$ on $\hat{X}_1$ and flips out three new rational curves on $\hat{X}_2$. Then one contracts the whole component $\hat{X}_1$. Keeping track of the component $\hat{X}_2$ in this procedure, we blow-up $\hat{X}_2$ along three distinct collinear points, and contract the induced $(-2)$-curve. Thus the central fiber is a quintic del Pezzo surface with an $A_1$-singularity. Tracking the boundary divisor part, one finds that the limit of $D_1$ is exactly $D_n$, and hence one constructs a special degeneration of $(\Sigma_5,D_1)$ to $(X_n,D_n)$.
\end{proof}

\begin{remark}
    \textup{Recall that a special test configuration is equivalent to a special valuation (cf. \cite[Remark 5.7]{Xu21}). The degeneration constructed in the proof is in fact the special test configuration induced by the divisorial valuation $\ord_{l_1}$ over $\Sigma_5$. }
\end{remark}

\emph{Proof of Theorem \ref{1}.} By openness of the K-semistability (cf. Theorem \ref{35}), it follows from the above two lemmas that $(\Sigma_5,cD_{\frac{1}{17}})$ is K-semistable for $c=1/17$. Thus the pair $(\Sigma_5,cD_{\frac{1}{17}})$ is K-stable for $c<1/17$ by interpolation (cf. Theorem \ref{36}). Combining with Lemma \ref{2}, one concludes that $(\Sigma_5,cD_{\frac{1}{17}})$ is K-semistable if and only if $0<c\leq\frac{1}{17}$.~\\

Now we show that $c=1/17$ is the first wall.

\begin{prop}\label{5}
Let $c\in(0,1/17)$ be a rational number and $(X,cD)$ be a K-semistable pair in $\ove{M}^K(c)$. Then $X\simeq \Sigma_5$.
\end{prop}

\begin{proof}
Since the smooth quintic del Pezzo surface is unique, then it suffices to show that $X$ is smooth for $0<c<1/17$. 
Assume that $x\in X$ is a singularity. Then locally we have $(x\in X)\simeq(\mb{C}^2/G)$ for some nontrivial finite group $G$. As $(X,cD)$ is K-semistable, then \cite[Proposition 4.6]{LL19} implies that $$5(1-2c)^2=(-K_X-cD)^2\leq \left(\frac{3}{2}\right)^2\widehat{\vol}(x,X,D)\leq\frac{9}{4}\widehat{\vol}(x,X).$$ Applying \cite[Theorem 2.7]{LX19}, one deduces that $$|G|\leq 4\cdot\frac{9}{4}\cdot\frac{1}{5(1-2c)^2}\leq\frac{9\cdot 17^2}{5\cdot 15^2}<3,$$ and thus $G\simeq \mb{Z}/2\mb{Z}$ and $x$ is an $A_1$-singularity. By classification of ADE del Pezzo surfaces in Section \ref{102}, either $X=X_1$ or $X=X_{1,1}$. In the former case, we have already seen in proof of Lemma \ref{104} that $\beta_{(X_{1},cD)}(E)\geq 0$ implies that $c\geq 1/17$, where $E$ is the exceptional divisor obtained by blowing up the singularity. By the same computation for the surface $X_{1,1}$ in Section \ref{108}, we conclude that $X$ has to be smooth.

\end{proof}

\subsection{After the first wall}

\begin{lemma}\label{106}
        Let $c_1=\frac{1}{17}$, and $D\in|-2K_{X_n}|$ be an effective divisor. If the $A_1$-singularity $p\in X_n$ is contained in $D$, then $(X_n,c_1D)$ is K-unstable. If $p\notin D$, then $(X_n,c_1D)$ is K-semistable with the K-polystable degeneration $(X_n,c_1D_n)$. In particular, the only $c_1$-K-polystable curve on $X_n$ is $D_n$.
\end{lemma}
    
\begin{proof}
    If $p\in D$, then the $\beta$-invariant of the pair with respect to the exceptional divisor $E$ over $p$ is $$\beta_{(X_n,cD)}(E)\leq 1-c-\frac{17(1-2c)}{15}=\frac{19c-2}{15}<0$$ for $c=\frac{1}{17}+\varepsilon$. This proves the first statement. If $p\notin D$, then the $\mb{G}_m$-action on $X_n$ induces a special degeneration of $(X_n,c_1D)$ to $(X_n,c_1D_n)$, and hence $(X_n,c_1D)$ is K-semistable by openness (cf. Theorem \ref{35})
\end{proof}

\begin{theorem} Let $0<c<\frac{1}{2}$ and $0<\varepsilon\ll 1$ be rational numbers. 
    \begin{enumerate}
        \item We have an isomorphism $\phi^{-}_{\frac{1}{17}}:\ove{M}^K\left(\frac{1}{17}-\varepsilon\right)\stackrel{\sim}{\rightarrow}\ove{M}^K\left(\frac{1}{17}\right)$.
        \item If $D\in|-2K_{X_n}|$ does not contain the $A_1$-singularity, and $(X_n,\frac{1}{2}D)$ is klt (e.g. $D$ is at worst nodal), then $(X_n,cD)$ is K-stable for any $\frac{1}{17}<c<\frac{1}{2}$. 
        \item The wall-crossing morphism $$\phi^{+}_{\frac{1}{17}}:\ove{M}^K\left(1/17+\varepsilon\right)\stackrel{\sim}{\rightarrow}\ove{M}^K\left(1/17\right)$$ is a divisorial contraction. Moreover, the exceptional divisor $E^{+}_{\frac{1}{17}}$ of $\phi^{+}_{\frac{1}{17}}$ is isomorphic to a finite quotient of a weight projective space.
    \end{enumerate}
\end{theorem}

\begin{proof}
   By Lemma \ref{106}, we know that if $(X_n,\frac{1}{4}D)$ is K-polystable, then $D=D_n$. Thus the exceptional locus of $\phi^{-}_{\frac{1}{17}}$ is a point, and hence an isomorphism. Moreover, if $p\notin D$ and $(X_n,\frac{1}{2}D)$ is klt, then $(X_n,cD)$ is K-stable for any $\frac{1}{17}<c<\frac{1}{2}$ by Lemma \ref{106} and interpolation. This proves (1) and (2).
 
   Notice that $X_n$ can be embedded in to $\mb{P}^1_{(u:v)}\times\mb{P}^2_{(x:y:z)}$ as a hypersurface defined by the equation $uyz+vx(y+z)=0$, and a curve of the class $-2K_{X_n}$ comes from $|\mtc{O}_{X_n}(2,2)|$. The $\mb{G}_m$-action $\lambda$ on $X_n$ is given by $$t((u:v),(x:y:z))=((tu:t^{-1}v),(t^2x:y:z)).$$ The singularity of $X_n$ is $p=((1:0),(1:0:0))$. The condition $p\in D$ comes down to saying that the the coefficient of $x^2u^2$ in any defining equation of $D$ is non-zero. Notice that we have an exact sequence of groups $$0\longrightarrow \Aut^0(X_n)=\mb{G}_m\longrightarrow \Aut(X_n)\longrightarrow G\longrightarrow0,$$ where $G$ is a finite group (cf. \cite[Section 8]{CP20}). Let $V$ be the affine subspace of $H^0(X_n,\mtc{O}_{X_n}(2,2))$ consisting of polynomials whose coefficient of $x^2u^2$ is $1$. Viewing $V$ as a vector space, the $\mb{G}_m$-action $\lambda$ on $X_n$ induces a $\mb{G}_m$-action on $V$. Taking the weight decomposition, we have that $$V=\bigoplus_{i=1}^4V_i,$$ where $\mb{G}_m$ acts on $V_i$ of weight $i$, and the dimensions of $V_i$'s are $2,5,5,3$ respectively. Then $[(V-\{0\})/\mb{G}_m]$ is a weighted projective stack, which admits a good moduli space $\mb{P}(1^2,2^5,3^5,4^3)$. There exists a universal family $$\pi:(\mtf{X},c\mtf{D})\rightarrow[(V-\{0\})/\mb{G}_m],$$ where the fiber of $f\in [(V-\{0\})/\mb{G}_m]$ is given by $(X_n,c(x^2u^2-f=0))$. Moreover, the finite group $G$ acts on $(\mtf{X},c\mtf{D})$ and $[(V-\{0\})/\mb{G}_m]$ such that $\pi$ is $G$-equivariant. By Lemma \ref{106}, each pair in $E^{+}_{\frac{1}{17}}$ is isomorphic to one of the fiber in the family $(\mtf{X},c\mtf{D})$. Thus by the same argument in the proof of Theorem \ref{55}, we prove that $$E^{+}_{\frac{1}{17}}\simeq \mb{P}(1^2,2^5,3^5,4^3)/G.$$
   
\end{proof}

\begin{remark}\label{61}\textup{
If $C$ is a non-special genus $6$ curve with four $g^2_6$, then $C$ is canonically embedded in $X_n$ of the class $-2K_{X_n}$ which does not contain the  $A_1$-singularity. Thus $[C]$ lies in the image of the rational map $E^{+}_{\frac{1}{14}}\dashrightarrow M_6$ defined by $[(X,cC)]\rightarrow [C]$.}
\end{remark}

\section{Walls corresponding to nodal del Pezzo surfaces}

The goal this section is to determine the walls given by quintic del Pezzo pairs where the surfaces have du Val singularities. 

Recall that we classify singular quintic del Pezzo surfaces with ADE singularities in Section \ref{102}. Among all of them, there are four $\mb{T}$-surfaces of complexity one, and two toric surfaces. Each complexity one surface corresponds to two walls. We will give detailed proof for the surface with an $A_1$-singularity in Section \ref{107}, and list all the other three cases. The toric case is more complicated, we will treat those two toric surfaces in Section \ref{108} and \ref{109}.

\subsection{ADE surfaces of complexity one}\label{107}

Let $\Sigma_5$ be the smooth del Pezzo surface obtained by blowing up $\mb{P}^2$ along $4$ general points $p_1,...,p_4$ with exceptional divisors $E_1,...,E_4$ respectively. Let $L_1$ and $L_2$ be the proper transforms of the lines connecting $p_1,p_2$ and $p_3,p_4$ respectively which intersect at $p$. Let $D_{\frac{2}{19}}:=3L_1+L_2+C+E_1+E_2\in |-2K_{\Sigma_5}|$ be a divisor, where $C$ is the proper transform of a general conic passing through $p_3,p_4$. Let $X_1$ be the quintic del Pezzo surface with one $A_1$-singularity. Keep the same notation as in the construction of $X_1$ above Lemma \ref{104}. Let $D'_{\frac{2}{19}}=3l_1+l_2+F_1+F_2+l_3+l_4$ be a divisor in $|-2K_{X_1}|$, where and $l_3$ and $l_4$ are the proper transforms of two general lines connecting the point $p_3$ and the line $\ove{p_1p_2}$.

\begin{prop}\label{110}
Let $0<c<\frac{1}{2}$ be a rational number, and $P\in X_1$ be the $A_1$-singularity.
    \begin{enumerate}
        \item The log Fano pair $(X_1,cD'_{\frac{2}{19}})$ is K-polystable if and only if $c=\frac{2}{19}$.
        \item The log Fano pair $(\Sigma_5,cD_{\frac{2}{19}})$ admits a special degeneration to $(X_1,cD'_{\frac{2}{19}})$.
        \item The pairs $(\Sigma_5,cD_{\frac{2}{19}})$ is K-semistable if and only if $0<c\leq\frac{2}{19}$.
        \item  Let $C\in|-2K_{X_1}|$ be a curve with at worst nodal singularity and $P\in C$ with $\mult_P(C)=1$. Then $(X_1,cC)$ is K-unstable for any $0<c<\frac{2}{19}$ and K-semistable for $c=\frac{2}{19}$. Moreover, if $(X_1,\frac{1}{2}C)$ is klt (e.g. $C$ has nodal singularities), then $(X_1,cC)$ is K-stable for any $\frac{2}{19}<c<\frac{1}{2}$.
        \item  Let $C\in|-2K_{X_1}|$ be a curve such that $\mult_P(C)\geq 2$. Then $(X_1,cC)$ is K-unstable for any $0<c<\frac{1}{2}$.
    \end{enumerate}
\end{prop}

\begin{proof}
    For (1)-(4), we use the same argument as in the proof for the first wall $c_1=\frac{1}{17}$. For (5), it suffices to observe that the $\beta$-invariant of the pair $(X_1,cC)$ with respect to the exceptional divisor $F$ over $P$ is $$\beta_{(X_1,cC)}(F)\leq 1-2c-\frac{17(1-2c)}{15}=-\frac{2(1-2c)}{15}<0$$ for any $0<c<\frac{1}{2}$.
\end{proof}

\begin{remark}
    \textup{There are two immediate corollaries. 
    \begin{enumerate}
        \item There are no other walls corresponding to the surface $X_1$, in the sense that there does not exist a wall $c=c_i$ which is not equal to $\frac{1}{17}$ or $\frac{2}{19}$, such that there is some $C\in|-2K_1|$ satisfying that $(X_1,cC)$ is K-unstable for $0<c<c_i$ and $(X_1,c_iC)$ is K-semistable.
        \item The image of the rational map $\ove{M}^K(1/2-\varepsilon)\dashrightarrow M_6$ contains the locus of smooth genus six curves with exactly four $g^2_6$.
    \end{enumerate}
    We can compute the dimension of the exceptional locus of the wall-crossing morphism $\ove{M}^K(\frac{2}{19}+\varepsilon)\rightarrow\ove{M}^K(\frac{2}{19})$ given by the pairs with one $A_1$-singularity. The dimension of divisors $D$ in $|-2K_{X_1}|$ such that $\mult_P(D)=1$ is $14$, and $X_1$ has a $1$-dimensional automorphism group. Thus the dimension of the exceptional locus is $13$. In particular, this is not a divisor in the K-moduli space.
    }
\end{remark}

Doing the same computations for the other surfaces of complexity one, we obtain the following table of walls. Moreover, these are all the walls corresponding to these surfaces, and none of them contributes to divisorial contractions.

\begin{center}
\renewcommand*{\arraystretch}{1.2}
\begin{table}[ht]
    \centering
      \begin{tabular}{ |c  |c |c|c|}
    \hline
     wall & surface & surface singularities & restriction on curves $C$    \\ \hline 
     
     $\frac{1}{17} $  &  $X_1$  & $A_1$-singularity $P$  &   $P\notin C$
     \\ \hline
     
     $\frac{2}{19} $  &  $X_1$  & $A_1$-singularity $P$  &   $\mult_P(C)=1$
     \\ \hline     
     
     $\frac{1}{7} $  &  $X_2$  & $A_2$-singularity $P$  &   $P\notin C$
     \\ \hline     
     
     $\frac{2}{9} $  &  $X_2$  & $A_1$-singularity $P$  &   $\mult_P(C)=1$
     \\ \hline     
     
     $\frac{2}{9} $  &  $X_3$  & $A_3$-singularity $P$  &   $P\notin C$
     \\ \hline
     
     $\frac{2}{7} $  &  $X_4$  & $A_4$-singularity $P$  &   $P\notin C$
     \\ \hline
          
     $\frac{4}{13} $  &  $X_3$  & $A_3$-singularity $P$  &   $\mult_P(C)=1$
     \\ \hline
     
     $\frac{4}{11} $  &  $X_4$  & $A_3$-singularity $P$  &   $\mult_P(C)=1$
     \\ \hline   
   \end{tabular}
    \caption{Walls corresponding to ADE del Pezzo of complexity one}
    \label{Kwall1}
\end{table}
\end{center}

\subsection{The surface $X_{1,1}$}\label{108}

Let $\Sigma_5$ be the smooth del Pezzo surface obtained by blowing up $\mb{P}^2$ along $4$ general points $p_1,...,p_4$ with exceptional divisors $E_1,...,E_4$ respectively. Let $L_1$ and $L_2$ be the proper transforms of the lines connecting $p_1,p_2$ and $p_3,p_4$ respectively which intersect at $p$. Consider the divisor in $G_{\frac{2}{19}}:=2L_1+2L_2+L_3+L_4\in|-2K_{\Sigma_5}|$, where and $L_3,L_4$ are two lines passing through $p$ which are not both $L_1$ or $L_2$. 

Let $X_{1,1}$ be the quintic del Pezzo surface with two $A_1$-singularities, obtained by taking weighted blow-up of $\mb{P}^2_{x,y,z}$ at $(0:1:0)$ of weight $(2,1)$ for $(x,z)$ and at $(1:0:0)$ of weight $(2,1)$ for $(y,z)$. Denote the two singularities over $(0:1:0)$ and $(1:0:0)$ by $P$ and $Q$, respectively. Let $F_1$ and $F_2$ be the exceptional divisors of $\pi:X_{1,1}\rightarrow\mb{P}^2$ containing $P$ and $Q$ respectively. Let $l_1$, $l_2$, $l$ be the proper transforms of the lines $\{x=0\}$, $\{y=0\}$ and $\{z=0\}$, respectively. Let $q=(0:0:1)$ be the intersection point of $l_1$ and $l_2$, and $l_3,l_4$ be the pull-back of lines $\{a_3x+b_3y=0\}$ and $\{a_4x+b_4y=0\}$ . Then $G'_{\frac{2}{19}}:=2l_1+2l_2+l_3+l_4$ is a section in $|-2K_{X_{1,1}}|$.

\begin{prop}\label{112}
Let $0<c<\frac{1}{2}$ be a rational number.
    \begin{enumerate}
        \item If $l_3+l_4=\pi^{*}(xy=0)$, then $(X_{1,1},cG'_{\frac{2}{19}})$ is K-polystable for $c=\frac{2}{19}$, and K-unstable for any other $0<c<\frac{1}{2}$.
        \item If $l_3+l_4=\pi^{*}(x(x+ay)=0)$ or $l_3+l_4=\pi^{*}(y(y+ax)=0)$ for some $a\in\mb{C}^{*}$, then $(X_{1,1},cG'_{\frac{2}{19}})$ is strictly K-semistable for $c=\frac{2}{19}$, and K-unstable for any other $0<c<\frac{1}{2}$.
        \item If $l_3+l_4=\pi^{*}((x+ay)(x+by)=0)$ for some $a,b\in\mb{C}^{*}$, then $(X_{1,1},cG'_{\frac{2}{19}})$ is K-polystable for $c=\frac{2}{19}$, and K-unstable for any other $0<c<\frac{1}{2}$.
        \item If $l_3+l_4=\pi^{*}(x^2=0)$ or $l_3+l_4=\pi^{*}(y^2=0)$, then $(X_{1,1},cG'_{\frac{2}{19}})$ K-unstable for any $0<c<\frac{1}{2}$.
    \end{enumerate}
\end{prop}

\begin{proof} We apply the criterion for $\mb{T}$-varieties.
    \begin{enumerate}
    \item For simplicity, denote the divisor $G'_{\frac{2}{19}}$ in this case by $\bf{G}'_{\frac{2}{19}}$. Notice that $(X_{1,1},c\bf{G}'_{\frac{2}{19}})$ is a toric pair. It follows from \cite[Theorem 1.24]{CA23} that $(X_{1,1},c\bf{G}'_{\frac{2}{19}})$  is K-polystable for $c=\frac{2}{19}$, and K-unstable for any other $0<c<\frac{1}{2}$.
    \item In this case, the pair $(X_{1,1},cG'_{\frac{2}{19}})$ is of complexity one, which admits a special degeneration to $(X_{1,1},c\bf{G}'_{\frac{2}{19}})$. Then by openness of K-semistability and (1), we conclude that $(X_{1,1},cG'_{\frac{2}{19}})$ is strictly K-semistable for $c=\frac{2}{19}$, and K-unstable for any other $0<c<\frac{1}{2}$.
    \item In this case, we can check K-polystablility using the criterion for complexity one pairs as in the proof of Lemma \ref{104}.
    \item One can prove by computing the $\beta$-invariant of the pair with respect to either $l_1$ or $l_2$.
\end{enumerate}
\end{proof}

\begin{prop}\label{111}
Let $0<c<\frac{1}{2}$ be a rational number. 
    \begin{enumerate}
        \item The log Fano pair $(\Sigma_5,cG_{\frac{2}{19}})$ admits a special degeneration to $(X_{1,1},cG'_{\frac{2}{19}})$ for some $G'_{\frac{2}{19}}$.
        \item The pairs $(\Sigma_5,cG_{\frac{2}{19}})$ is K-semistable if and only if $0<c\leq\frac{2}{19}$.
        \item  Let $C\in|-2K_{X_{1,1}}|$ be a curve, which comes from the plane sextic curve $\wt{C}$ whose defining equation is $$f_6(x,y)+zf_5(x,y)+\cdots+z^6f_0(x,y),$$ where $f_6(x,y)$ is non-zero and not divided by either $x^4$ or $y^4$. Then $(X_{1,1},cC)$ is K-unstable for any $0<c<\frac{2}{19}$ and K-semistable for $c=\frac{2}{19}$. Moreover, if $(X_{1,1},\frac{1}{2}C)$ is klt, then $(X_{1,1},cC)$ is K-stable for any $\frac{2}{19}<c<\frac{1}{2}$.
    \end{enumerate}
\end{prop}

\begin{proof}
    The same proof as for Proposition \ref{110} applies here. We only add an explanation for (3). Since $f_6(x,y)$ is non-zero, then it has to be of the form $x^2y^2g_2(x,y)$ for some non-zero degree $2$ polynomial $g_2(x,y)$. Thus the $\mb{G}_m$-action given by $$x\mapsto tx,\quad y\mapsto ty$$ induces a special degeneration of $(X_{1,1},cC)$ to $(X_{1,1},cG'_{\frac{2}{19}})$. As $f_6(x,y)$ is not divided by $x^4$ or $y^4$, then $(X_{1,1},\frac{2}{19}G'_{\frac{2}{19}})$ is K-polystable. This proves that $(X_{1,1},\frac{2}{19}C)$ is K-semistable. By computing the $\beta$-invariant of $(X_{1,1},cC)$ with respect to $l$, we see that $$\beta_{(X_{1,1},cC)}(l)= 1-\frac{19(1-2c)}{15}=-\frac{2(19c-2)}{15}<0$$ for any $0<c<\frac{2}{19}$.
\end{proof}

In fact, the condition that $C\in|-2K_{X_{1,1}}|$ implies that the defining polynomial of the corresponding $\wt{C}$ satisfies that $$f_5(x,y)=x^2y^2g_1(x,y),\quad f_4(x,y)=xyg'_2(x,y),\quad \textup{and}\quad f_3(x,y)=xyg'_1(x,y).$$ Moreover, if $g_1$ and $g_2$ are both zero, then the pair $(X_{1,1},cC)$ is K-unstable for any $0<c<\frac{1}{2}$. Thus the similar computation gives us other walls corresponding to the surface $X_{1,1}$ (cf. Table \ref{Kwall2}).

\begin{center}
\renewcommand*{\arraystretch}{1.2}
\begin{table}[ht]
    \centering
      \begin{tabular}{ |c  |c |c|c|}
    \hline
     wall & curves & weight & $\beta$-invariant    \\ \hline 
     
     $\frac{2}{19} $  &  $x^2y^2(x-ay)(y-bx)$  & $(1,1,0)$  &   $2(19c-2)/15$
     \\ \hline
     
     $\frac{1}{7} $  &  $x^2y^2(x^2-yz)$  & $(1,2,0)$  &   $(7c-1)/5$
     \\ \hline     
     
     $\frac{4}{23} $  &  $xy^2(z^2y-x^3)$  & $(1,3,0)$  &   $2(23c-4)/15$
     \\ \hline     
     
     $\frac{2}{9} $  &  $xy^2z(x^2-yz)$  & $(1,2,0)$  &   $(9c-2)/5$
     \\ \hline

   \end{tabular}
    \caption{Walls corresponding to $X_{1,1}$}
    \label{Kwall2}
\end{table}
\end{center}

\subsection{The surface $X_{1,2}$}\label{109}

Let $\Sigma_5$ be the smooth del Pezzo surface obtained by blowing up $\mb{P}^2$ along $4$ general points $p_1,...,p_4$ with exceptional divisors $E_1,...,E_4$ respectively. Let $L$ be the proper transforms of the line passing through $p_1,p_2$, and $C_1,C_2$ be the proper transform of two conics passing through $p_3,p_4$ and tangent to $L$ at a point $p$ such that not both of them are singular. Then $D_{\frac{4}{23}}:=2L+C_1+C_2$ is a section in $|-2K_{\Sigma_5}|$.

Let $X_{1,2}$ be the quintic del Pezzo surface with an $A_1$-singularity and an $A_2$-singularity obtained as following. Blow up $\mb{P}^2_{(x:y:z)}$ along two tangent vectors: the one supported on $(0:1:0)$ along $\{x=0\}$ and the one supported on $(0:0:1)$ along $\{y=0\}$. Let $F_1,F_2$ be the exceptional divisors over $(0:1:0)$ with $F_1^2=-2$, $F_2^2=-1$, and $F_3,F_4$ be the exceptional divisors over $(0:0:1)$ with $F_3^2=-2$, $F_4^2=-1$. Then contract the three $(-2)$-curves: $F_1,F_3$ and the proper transform of $\{x=0\}$. Then the resulting surface, denoted by $X_{1,2}$, has an $A_1$-singularity $P$ and an $A_2$-singularity $Q$. Let $C'_1,C'_2$ be the proper transformation of two conics $\{x^2-ayz=0\}$ such that not both of them are singular, and $L_y$ be the proper transform of the line $\{y=0\}$. Then $D'_{\frac{4}{23}}=C'_1+C'_2+2L'$ is a section in $|-2K_{X_{1,2}}|$.

\begin{prop}\label{114}
Let $0<c<\frac{1}{2}$ be a rational number.
    \begin{enumerate}
        \item The log Fano pair $(X_{1,2},cD'_{\frac{4}{23}})$ is K-polystable if and only if $c=\frac{4}{23}$.
        \item The log Fano pair $(\Sigma_5,cD_{\frac{4}{23}})$ admits a special degeneration to $(X_{1,2},cD'_{\frac{4}{23}})$.
        \item The pairs $(\Sigma_5,cD_{\frac{4}{23}})$ is K-semistable if and only if $0<c\leq\frac{4}{23}$.
        \item Let $C\in|-2K_{X_{1,2}}|$ be a curve such that either $\mult_{P}(C)\geq2$ or $\mult_{Q}(C)\geq2$. Then the pair $(X_{1,2},cC)$ is K-unstable for any $0<c<\frac{1}{2}$.
        \item  Let $C\in|-2K_{X_{1,2}}|$ be a curve such that and $P,Q\notin C$. Then $(X_{1,2},cC)$ is K-unstable for any $0<c<\frac{4}{23}$ and K-semistable for $c=\frac{4}{23}$. Moreover, if $(X_{1,2},\frac{1}{2}C)$ is klt, then $(X_{1,2},cC)$ is K-stable for any $\frac{4}{23}<c<\frac{1}{2}$.
    \end{enumerate}
\end{prop}

\begin{proof}
    The first three statements are proved similarly as in Section \ref{107} and \ref{108}. For (4), one has that $A_{(X_{1,2},cC)}(F_3)=1$ for the divisor $F_3$ over the singularity $Q$, and that  
    \begin{equation}
\begin{split}
    S_{(X_{1,2},cC)}(F_3)&=\frac{1-2c}{5}\int_0^{4}\vol(-K_{X_{1,2}}-tF_3)dt\\
    &=\frac{1-2c}{5}\left(\int_0^{3}5-2t+\frac{t^2}{6}dt+\int_{3}^42(2-t/2)^2dt\right)\\
    &=\frac{23(1-2c)}{15}.
\end{split}
\end{equation} so that $\beta_{(X_{1,2}.cC)}(F_3)=\frac{2(23c-4)}{15}<0$ for any $0<c<\frac{4}{23}$. Thus $(X_{1,1},cC)$ is K-unstable for any $0<c<\frac{4}{23}$. Let $\lambda$ be the $\mb{G}_m$-action on the pair $(X_{1,2},cD'_{\frac{4}{23}})$, which is induced from the $\mb{G}_m$-action on $\mb{P}^2$ defined by $$t\cdot (x:y:z):=(x:t^2y:tz).$$ We claim that the $$\lim_{t\to 0}t\cdot(X_{1,2},cC)=(X_{1,2},cD'_{\frac{4}{23}}).$$ Notice that $C$ comes from a plane sextic curve $\wt{C}\subseteq \mb{P}^2_{(x:y:z)}$  which contains through the locus to blow up of multiplicity at least two. Suppose $\wt{C}$ is defined by a polynomial $$f(x,y,z)=z^6f_0(x,y)+z^5f_1(x,y)+\cdots+f_6(x,y).$$ Since $Q\notin C$, then we have that $f_0=f_1=0$, $f_2(x,y)=ay^2$ for some $a\in \mb{C}^{*}$, and the coefficients of $yx^2$ and $x^3$ in $f_3(x,y)$ are zero. The condition $C\in |-2K_{X_{1,2}}|$ also implies that each monomial in $f(x,y,z)$ with non-zero coefficient cannot be divided by $y^5$. Since $P\notin C$, then the coefficient of some monomial divided by $y^4$ is non-zero, and this monomial has to be $y^4x^2$. In this case, we also have that the coefficient of $y^3$ in $f_3(x,y)$ is zero. Thus we have that $$\lim_{t\to 0}t\cdot f(x,y,z)=y^2(xy-a_1z^2)(xy-a_2z^2)$$ up to a constant scaling, where $a_1,a_2\in\mb{C}^{*}$. Therefore, by openness of K-semistability, we have that $(X_{1,2},\frac{4}{23}C)$ is K-semistable.
\end{proof}

Similar argument as in Section \ref{108} gives us all the walls corresponding to $X_{1,2}$ (cf. Table \ref{Kwall3}).

\begin{center}
\renewcommand*{\arraystretch}{1.2}
\begin{table}[ht]
    \centering
      \begin{tabular}{ |c  |c |c|c|}
    \hline
     wall & curves & weight & $\beta$-invariant    \\ \hline 
     
     $\frac{4}{23} $  &  $y^2(xy-z^2)(axy-z^2)$  & $(0,2,1)$  &   $2(23c-4)/15$
     \\ \hline
     
     $\frac{1}{5} $  &  $y^2z(z^3-x^2y)$  & $(0,3,1)$  &   $2(5c-1)/3$
     \\ \hline     
     
     $\frac{2}{9} $  &  $y^2(yx^3-z^4)$  & $(0,4,1)$  &   $2(9c-2)/5$
     \\ \hline     
     
     $\frac{2}{9} $  &  $xy^2z^2(y-z)$  & $(0,1,1)$  &   $(23c-4)/15$
     \\ \hline     
     
     $\frac{7}{29} $  &  $y^2xz(xy-z^2)$  & $(0,2,1)$  &   $2(29c-7)/15$
     \\ \hline
     
     $\frac{8}{31} $  &  $y^2x(xy^2-z^3)$  & $(0,3,2)$  &   $(31c-8)/15$
     \\ \hline
          
     $\frac{2}{7} $  &  $yxz^2(y^2-xz)$  & $(0,1,2)$  &   $(7c-2)/3$
     \\ \hline
     
     $\frac{2}{7} $  &  $xy^2(z^3-yx^2)$  & $(0,3,1)$  &   $(7c-2)/3$
     \\ \hline   
   \end{tabular}
    \caption{Walls corresponding to $X_{1,2}$}
    \label{Kwall3}
\end{table}
\end{center}

\begin{remark}\label{203}
    \textup{For any wall $c=c_i\neq \frac{1}{17}$ corresponding to ADE del Pezzo surfaces, the wall-crossing morphism $\ove{M}^K(c_i+\varepsilon)\rightarrow\ove{M}^K(c_i)$ is a flipping contraction. In particular, they do not contribute the Picard number of $\ove{M}^K(c)$. Also, for every smooth genus six curve $C$ with finitely many $g^2_6$, there is a K-stable log Fano pair $(X,cC)$, where $c=\frac{1}{2}-\varepsilon$ and $X$ is an ADE quintic del Pezzo surface.}
\end{remark}

\section{The second divisorial contraction: trigonal curves}

In this section, we study the second divisorial contraction $\ove{M}^K(c_i+\varepsilon)\rightarrow \ove{M}^K(c_i)$ occurring at the wall $c_i=\frac{11}{52}$. In particular, a general K-polystable pair $(X,cC)\in\ove{M}^K(c)$ on the exceptional divisors has boundary divisor $C$ a genus six trigonal curve.

A general genus six smooth trigonal curve $C$ is canonically embedded into $\mb{P}_{x,y}^1\times\mb{P}_{u,v}^1$ as a divisor of class $\mtc{O}_{\mb{P}^1\times\mb{P}^1}(3,4)$. Let $L_x$ and $L_y$ be the rulings defined by $\{x=0\}$ and $\{y=0\}$ respectively. Assume that $L_x$ is not a component of $C$. Blow up $\mb{P}^1\times\mb{P}^1$ along the length $4$ subscheme $L_x\cap C$, and contract the proper transform of $L_x$. Contract also the $(-2)$-curves if $L_x$ does not intersect $C$ transversely. Denote the induced surface by $X$. One sees that $X$ is a $\mb{Q}$-Fano surface of degree $5$, which is isomorphic to the blow-up of $\mb{P}(1,1,4)_{p,q,r}$ along four points (not necessarily distinct) on the infinity section $\{r=0\}$. Moreover, the proper transform of $C$ on $X$ is of the divisor class $-2K_{X}$.

Conversely, let $X$ be the blow-up of $\mb{P}(1,1,4)$ along four points on the infinity section, and $C\in|-2K_X|$ is an effective divisor not passing through the $\frac{1}{4}(1,1)$-singularity $P$. Let $l_1,...,l_4$ be the proper transforms on $X$ of the rulings passing through the four blown-up points on $\mb{P}(1,1,4)$. Blowing up $X$ at $P$ and contracting the proper transforms of $l_1,...,l_4$, one obtains a smooth quadric $\mb{P}^1\times \mb{P}^1$. Moreover, the proper transform of $C$ on $\mb{P}^1\times \mb{P}^1$ is a divisor of the class $\mtc{O}_{\mb{P}^1\times \mb{P}^1}(3,4)$.

\subsection{Trigonal curves in $\mb{F}_0$}

In this subsection, we study the first wall $c=\frac{11}{52}$ where blow-up of $\mb{P}(1,1,4)$ together with trigonal curves appear.

Let $X$ be an ADE quintic del Pezzo surface obtained by blowing up $\mb{P}^2$ along $4$ points $p_1,...,p_4$ (not necessarily distinct), let $D_{\frac{11}{52}}\in|-2K_X|$ be a divisor obtained from a triple conic $3Q$ on $\mb{P}^2$ passing through $p_1,...,p_4$. 

Let $X_t$ ("t" denotes "trigonal") be a degree $5$ del Pezzo surface obtained by blowing up $\mb{P}_{p,q,r}(1,1,4)$ along $4$ points $q_1,...,q_4$ on the infinity line $C'=\{r=0\}$. Here we allow $q_i=q_j$, but require that any three of them cannot collide. Let $D'_{\frac{11}{52}}=3C'+F_1+F_2+F_3+F_4$ be a section in $|-2K_{X_t}|$, where $F_i$ is the exceptional divisor over $q_i$. The $\mb{Q}$-Gorenstein degeneration of $\mb{P}^2$ to $\mb{P}(1,1,4)$ induces a $\mb{Q}$-Gorenstein degeneration of the pair $(X,cD_{\frac{11}{52}})$ to $(X_t,cD'_{\frac{11}{52}})$.

\begin{theorem}\label{3} Let $0<c<\frac{1}{2}$ be a rational number.
\begin{enumerate}
    \item The log Fano pair $(X_t,cD'_{\frac{11}{52}})$ is K-polystable if and only if $c=\frac{11}{52}$.
    \item The log Fano pair $(\Sigma_5,cD_{\frac{11}{52}})$ is K-stable if and only if $0<c<\frac{11}{52}$, and K-unstable for $\frac{11}{52}<c<\frac{1}{2}$. 
    \item Let $D\in|-2K_{X_t}|$ be a divisor. If $D$ contains the singularity $P$, then $(X_t,cD)$ is K-unstable for any $0<c<\frac{1}{2}$. If $D$ does not contain $P$, then $(X,cD)$ is  K-unstable for $0<c<\frac{11}{52}$, and $(X,\frac{11}{52}D)$ is K-semistable.
    \item The wall-crossing morphism $$\phi^{+}_{\frac{11}{52}}:\ove{M}^K\left(11/52+\varepsilon\right)\rightarrow \ove{M}^K\left(11/52\right)$$ is a divisorial contraction.
\end{enumerate}
\end{theorem}

\begin{proof}
    Notice that if at most two of $q_1,...,q_4$ collide, then the pair $(X_t,cD_t)$ is a complexity one $\mb{T}$-pair with the $\mb{G}_m$-actions $\lambda$ induced from the $\mb{G}_m$-actions on $\mb{P}(1,1,4)_{p,q,r}$ given by $$p\mapsto t\cdot p,\quad q\mapsto t\cdot q,\quad r\mapsto r.$$ If $q_1,...,q_4$ collide pairwise, then the pair $(X_t,cD_t)$ is toric. For simplicity, we only prove for the former case. The toric case can be proven using \cite[Theorem 1.21]{CA23}.
    
    Then the only horizontal divisor with respect to the $\mb{G}_m$-action $\lambda$  is $C'$, and all vertical divisors on $X_t$ are the lines joining the singularity $P$ and $C'$.

    We compute the $\beta$-invariant of the divisor $C'$. We have that $A_{(X_t,cD'_{\frac{11}{52}})}(C')=1-3c$ and that \begin{equation}
    \begin{split}
    S_{(X_t,cD'_{\frac{11}{52}})}(C')&=\frac{1-2c}{5}\int_0^{3/2}\vol(-K_{X_t}-tC)dt\\
    &=\frac{1-2c}{5}\left(\int_0^{1}5-4tdt+\int_1^{3/2}(3-2t)^2dt\right)\\
    &=\frac{19(1-2c)}{30}.
    \end{split}
    \end{equation}
    Thus $\beta_{(X_t,cD'_{\frac{11}{52}})}(C')=0$ if and only if $c=\frac{11}{52}$. Similarly, we can check that when $c=11/52$, the $\beta$-invariants of the vertical divisors are positive. It follows from Theorem \ref{71} that $(X_t,cD'_{\frac{11}{52}})$ is K-polystable if and only if $c=\frac{11}{52}$. By openness of K-semistability and interpolation, we also have that $(X,cD_{\frac{11}{52}})$ is K-stable for $0<c<\frac{11}{52}$, and K-semistable for $c=\frac{11}{52}$. When $c>\frac{11}{52}$, by computing $\beta_{(X,cD_{\frac{11}{52}})}(C)$, we can conclude that $(X,cD_{\frac{11}{52}})$ is K-unstable.

    Now we prove (3). Let $E$ be the exceptional divisor over $P$. We can similarly compute that $S_{(X_t,cD)}(E)=\frac{13(1-2c)}{15}$. If $P$ is contained in $D$, then $$\beta_{(X_t,cD)}(E)\leq \frac{1}{2}-c-\frac{13(1-2c)}{15}\leq0$$ for any $0<c<\frac{1}{2}$. If $P\notin D$, then the condition $\beta_{(X_t,cD)}(E)\geq0$ implies that $c\geq\frac{11}{52}$. Notice that the $\mb{G}_m$-action $\lambda$ on $X_t$ induces a special degeneration of $(X_t,cD)$ to $(X_t,cD'_{\frac{11}{52}})$. Thus by openness, we conclude that $(X_t,\frac{11}{52}D)$ is K-semistable. In particular, if $(X_t,\frac{1}{2}D)$ is klt (e.g. $D$ is smooth), then $(X_t,cD)$ is K-stable for any $\frac{11}{52}<c<\frac{1}{2}$.

    For (4), observe that $\dim \Aut(X_t)=2$, $\dim |-2K_{X_t}|=15$, and that $X_t$ has a one-dimensional moduli. Thus the exceptional locus of the wall crossing morphism $\phi^{+}_{\frac{11}{52}}$ is a divisor.
\end{proof}

\begin{remark}\label{62}\textup{
    If $(X,cD)\in E^{+}_{\frac{11}{52}}$ such that $D$ is smooth, then $D$ is a trigonal genus six curve. The locus $E^{+}_{\frac{11}{52}}$ admits a rational map onto the $13$-dimensional locus of trigonal curves in $\ove{M}_6$, which can be embedded into $\mb{P}^1\times\mb{P}^1$.
}\end{remark}

Recall that when constructing $X_t$, we have a restriction on the four points $q_1,...,q_4$ on $\mb{P}(1,1,4)$ to blow up. Now we deal with the case when at least three of them coincide. In this case, a defining polynomial of $\wt{C}$, the corresponding curve of $C\in|-2K_{X_t}|$, is $$y^3f_4(u,v)+y^2xg_4(u,v)+yx^2h_4(x,y)+x^3l_4(x,y),$$ where we may assume that $f_4(u,v)$ is either $u^3v$ or $u^4$.

\begin{prop} Let $0<c<\frac{1}{2}$ be a rational number, and $X_t$ be the blow-up of $\mb{P}(1,1,4)$ along four points on the infinity section, where three of them have the same support and curvilinear with respect to the infinity section.
\begin{enumerate}
    \item Let $C_0\in|-2K_{X_t}|$ be a curve corresponding to $\wt{C}=\{y^2v(yu^3-xv^3)\}$. Then $(X_t,cC_0)$ is K-polystable if and only if $c=\frac{17}{64}$.
    \item Let $C\in|-2K_{X_t}|$ be such that $f_4(u,v)=u^3v$ and $g_4(u,v)$ has non-zero $v^4$ term. Then $(X_t,cC)$ is K-unstable for $0<c<\frac{17}{64}$ and K-semistable for $c=\frac{17}{64}$.
\end{enumerate}
\end{prop}

\begin{proof}
    Since $(X_t,cC_0)$ is a $\mb{T}$-pair of complexity one, we can apply Theorem \ref{71}. The horizontal divisor on $X_t$ is the proper transform on $X_t$ of the ruling, which passes through the $A_2$-singularity of $X_t$. Denote it by $l$. We have that $$A_{(X_t,cC_0)}(l)=1,\quad \textup{and}\quad S_{(X_t,cC_0)}(l)=\frac{32(1-2c)}{15}.$$ Thus $\beta_{(X_t,cC_0)}(l)=0$ if and only if $c=\frac{17}{64}$. Similarly, one can check that $\beta_{(X_t,cC_0)}(E)>0$ for any other $\mb{G}_m$-invariant vertical divisor $E$ on $X_t$.

    The computation of $\beta$-invariant with respect to $l$ above shows that if $(X,cC)$ is K-semistable, then $c\geq \frac{17}{64}$. Moreover, the $\mb{G}_m$-action $\lambda$ induced from the one on $\mb{P}^1\times\mb{P}^1$ defined by $$x\mapsto t^{-3}x,\quad y\mapsto t^{3}y,\quad u\mapsto t^{-1}u,\quad v\mapsto tv$$ yields a special degeneration from $(X_t,cC)$ to $(X_t,cC_0)$. Thus $(X_t,\frac{17}{64}C)$ is K-semistable by openness.
\end{proof}

Applying the same argument, we can get the following table of walls corresponding to surfaces of type $X_t$. If a curve $C$ does not admit an isotrivial degeneration to one of the curves in the following table, then the pair $(X_t,cC)$ is K-unstable for any $0<c<\frac{1}{2}$.

\begin{center}
\renewcommand*{\arraystretch}{1.2}
\begin{table}[ht]
    \centering
      \begin{tabular}{ |c  |c |c|c|}
    \hline
     wall & curves & weight & $\beta$-invariant    \\ \hline 
     
     $\frac{17}{64} $  &  $y^2v(yu^3-xv^3)$  & $(-3,3;-1,1)$  &   $(64c-17)/15$
     \\ \hline

    $\frac{23}{76} $  &  $y^2uv(yu^2-xv^2)$  & $(-2,2;-1,1)$  &   $(76c-23)/15$
     \\ \hline 
     
     $\frac{23}{76} $  &  $y^2(yu^4-xv^4)$  & $(-4,4;-1,1)$  &   $(76c-23)/15$
     \\ \hline     

     $\frac{7}{20} $  &  $yuv(y^2u^3-x^2v^3)$  & $(-3,3;-2,2)$  &   $(20c-7)/3$
     \\ \hline 
     
     $\frac{7}{20} $  &  $y^2u(yu^3-xv^3)$  & $(-3,3;-1,1)$  &   $(20c-7)/3$
     \\ \hline

   \end{tabular}
    \caption{Walls corresponding to $X_{t}$}
    \label{Kwall4}
\end{table}
\end{center}

\subsection{Trigonal curves in $\mb{F}_2$}\label{117}

In this section, we will study the sublocus of our K-moduli spaces corresponding to the quintic del Pezzo surfaces of type (iii) in Theorem \ref{40}, and those genus six trigonal curves which are canonically embedded into the Hirzebruch surface $\mb{F}_2$.

Let $\Sigma_5$ be the smooth del Pezzo surface obtained by blowing up $\mb{P}^2$ along $4$ general points $p_1,...,p_4$ with exceptional divisors $E_1,...,E_4$ respectively. Let $L_1$ be a general line on $\mb{P}^2$, and $L_2$ be the line connecting $p_3,p_4$, which intersects with $L_1$ at the point $p$. Let $C_1,C_2$ be two conics that are tangent to $L_1$ at $p$ and passing through $p_1,p_2,p_3$ and $p_1,p_2,p_4$ respectively. Then the plane sextic curve $L_1+L_2+C_1+C_2$ induces a section in $|-2K_{\Sigma_5}|$, denoted by $D_{\frac{13}{41}}$.

Let $X'$ be the degree $5$ del Pezzo surface corresponding to the type (iii) in Theorem \ref{40}. Precisely, $X'$ is the anti-canonical ample model of the blow-up of $\mb{F}_2$ along three points $q_1,q_2,q_3$ on a fixed fiber outside the negative section, and along a general point on the negative section. Here, we require that at most two of $q_1,...,q_3$ can collide. The surface $X'$ acquires a singularity $P$ of type $\frac{1}{8}(1,3)$. Let $F_1,F_2,F_3$ be the exceptional divisors over points on the fiber and $G$ the one over the point on the section. Notice that such a surface admits a $\mb{G}_m$ action, where the $\mb{G}_m$-equivariant divisors are the fiber at infinity and the sections in the class $E+2F$, $E+2F-F_i$ or exceptional divisors $G$ and $F_i$. Let $C'_1,C'_2,C'_3$ be the three divisors $E+2F-F_i$ and $L'=F-G$ the fiber at infinity. Then $D'_{\frac{13}{41}}:=C'_1+C'_2+C'_3+L'$ is a section in $|-2K_{X'}|$. Notice that the log Fano pair $(X',cD'_{\frac{13}{41}})$ is a $\mb{T}$-pair of complexity one, thus we have the following result.

\begin{prop} Let $0<c<\frac{1}{2}$ be a rational number.
\begin{enumerate}
    \item The log Fano pair $(X',cD'_{\frac{13}{41}})$ is K-polystable if and only if $c=\frac{13}{41}$.
    \item If the pair $(\Sigma_5,cD_{\frac{13}{41}})$ is K-semistable, then $0<c\leq\frac{13}{41}$.
    \item The log Fano pair $(\Sigma_5,cD_{\frac{13}{41}})$ admits a degeneration to $(X',cD'_{\frac{13}{41}})$.
\end{enumerate}
    
\end{prop}

\begin{proof}
    The proof is similar to those statements in previous sections. For the reader's convenience, we put the computation of $\beta$-invariants here.

    Let $E$ be the exceptional divisor of the $(1,2)$-weighted blow-up of $\Sigma_5$ at $p$. Then we have that $A_{(\Sigma_5,cD_{\frac{13}{41}})}(E)=3-7c$ and that \begin{equation}
\begin{split}
    S_{(\Sigma_5,cD_{\frac{13}{41}})}(E)&=\frac{1-2c}{5}\int_0^{7/2}\vol(-K_{\Sigma_5}-tF)dt\\
    &=\frac{1-2c}{5}\left(\int_0^{2}5-\frac{t^2}{2}dt+\int_2^{3}3-\frac{(t-2)^2}{3}dt+\int_3^{7/2}\frac{8t^2-32t+50}{3}dt\right)\\
    &=\frac{32(1-2c)}{15}.
\end{split}
\end{equation} Thus $\beta_{(\Sigma_5,cD_{\frac{13}{41}})}(E)\geq0$ imposes the condition $c\leq 13/41$. 

Analogously, let $l$ be the fiber of $\mb{F}_2$ on which we blow up three points. We have that $A_{(X',cD'_{\frac{13}{41}})}(l)=1/2$ and that  \begin{equation}\nonumber
\begin{split}
    S_{(X',cD'_{\frac{13}{41}})}(l)&=\frac{1-2c}{5}\int_0^{7/2}\vol(-K_{X'}-tl)dt\\
    &=\frac{1-2c}{5}\left(\int_0^{1/2}5-\frac{8t^2}{3}dt+\int_{1/2}^{3/2}\frac{t^2-9t+69/4}{3}dt+\int_{3/2}^{7/2}\frac{(7-2t)^2}{8}dt\right)\\
    &=\frac{41(1-2c)}{30}.
\end{split}
\end{equation} so that $\beta_{(X'.cD'_{\frac{13}{41}})}(l)=0$ if and only if $c=\frac{13}{41}$.
\end{proof}

\begin{prop}\label{116}
    Let $0<c<\frac{1}{2}$ be a rational number, and $C\in|-2K_{X'}|$ be an effective divisor.
    \begin{enumerate}
        \item If $C$ contains the $\frac{1}{8}(1,3)$-singularity $P$ of $X'$, then $(X',cC)$ is K-unstable for any $0<c<\frac{1}{2}$.
        \item If $C$ does not contain $P$, then $(X',\frac{13}{41}C)$ is K-semistable. 
    \end{enumerate}
\end{prop}

\begin{proof}
    To show (1), one computes the $\beta$-invariant of the pair with respect to the exceptional divisor over $P$. For (2), it suffices to notice that if $P\notin C$, then the $\mb{G}_m$-action on $X'$ induces a special degeneration of $(X',\frac{13}{41}C)$ to $(X',\frac{13}{41}D'_{\frac{13}{41}})$.
\end{proof}

\begin{remark}\label{69}\textup{
\begin{enumerate}
    \item Across this wall, we recover the genus six trigonal curves which are canonically embedded into $\mb{F}_2$: it is clear that a smooth curve in $|-2K_{X'}|$ is such a curve. Conversely, let $C\subseteq\mb{F}_2$ be a genus six trigonal curve, then its class is $3E+7F$. In particular, $(E.C)=1$ and $(C.F)=3$. Take a general fiber intersecting $C$ at $q_1,q_2,q_3$. Then blowing up $\mb{F}_2$ at $q_1,q_2,q_3$ and $q:=C\cap E$, and contracting the two $(-3)$-curves, one get a pair in $\ove{M}^K\left(\frac{13}{41}+\varepsilon\right)$.
    \item Notice that $\dim \Aut(X')=3$, and these surfaces have a one-dimensional moduli space. Thus the dimension of the exceptional locus of $\ove{M}^K\left(\frac{13}{41}+\varepsilon\right)\rightarrow \ove{M}^K\left(\frac{13}{41}\right)$ is $13$. 
\end{enumerate}
}\end{remark}

As an immediate consequence, there are no other walls contributed by the appearance of pairs, whose surface part is $X'$. However, when constructing $X'$, we assume that at most two of the three points to blow-up on the fiber can collide. Now we deal with the case when the three points degenerate to a length three curvilinear subscheme. In this case, $X'$ has an extra $A_2$-singularity $Q$. A curve $C\in |-2K_{X'}|$ comes from an effective divisor on $\mb{F}_2$ of the class $3e+7f$, and hence a section $\wt{C}\in|\mtc{O}_{\mb{P}(1,1,2)}(7)|$. If $C$ contains $P$, as in Proposition \ref{116}(1), the pair $(X',cC)$ is K-unstable for any $0<c<\frac{1}{2}$. Similarly, if $\mult_Q(C)\geq2$, then $(X',cC)$ is also K-unstable (cf. Proposition \ref{110}(5)). Thus it suffices to consider the case when $P\notin C$ and $\mult_Q(C)\leq1$. 

Let $(x:y:z)$ be the coordinate of $\mb{P}(1,1,2)$. Then a curve $\wt{C}\in|\mtc{O}_{\mb{P}}(7)|$ is given by equation $$f_1(x,y)z^3+f_3(x,y)z^2+f_5(x,y)z+f_7(x,y)=0,$$ where $f_d(x,y)$ is a homogeneous polynomial in $x,y$ of degree $d$. The condition $P\notin C$ requires that $f_1(x,y)\neq0$. We may assume that $f_1(x,y)=x$. We can also assume that the length $3$ subscheme to blow up is supported at $(0:0:1)$, and curvilinear with respect to the line $\{y=0\}$. 

\begin{prop}
    Let $D_2\in|-2K_{X'}|$ be a divisor corresponding to the curve $\{xz(z^2-yx^3)=0\}$ in $|\mtc{O}_{\mb{P}(1,1,2)}(7)|$. 
    \begin{enumerate}
        \item The pair $(X',cD_2)$ is K-polystable if and only if $c=\frac{19}{53}$.
        \item If $C\in|-2K_{X'}|$ is a divisor such that $P\notin C$ and $\mult_Q(C)=1$, then $(X',cC)$ admits a special degeneration to $(X',cD_2)$.
    \end{enumerate}
\end{prop}

\begin{proof}
      We only give the proof for the degeneration, since the remaining part follows from the same argument as before. Consider the local equation of $\wt{C}$ at $(1:0:0)$, which is of the form $$ayz+by^2+z^3+cz^2y+\cdots=0.$$ Notice also that $a\neq0$, otherwise $C$ is not a section in $|-2K_{X'}|$. Then the $\mb{G}_m$-action $\lambda$ given by $y\mapsto t^2y$, $z\mapsto tz$ gives the desired degeneration.
\end{proof}

Similarly, we have the following result giving the other wall $c=\frac{16}{47}$, for which we use the $\mb{G}_m$-action $y\mapsto t^3y$, $z\mapsto tz$ to get the special degeneration.

\begin{prop}
    Let $D_1\in|-2K_{X'}|$ be a divisor corresponding to the curve $\{x(z^3-yx^5)=0\}$ in $|\mtc{O}_{\mb{P}(1,1,2)}(7)|$. 
    \begin{enumerate}
        \item The pair $(X',cD_1)$ is K-polystable if and only if $c=\frac{16}{47}$.
        \item If $C\in|-2K_{X'}|$ is a divisor such that $P,Q\notin C$, then $(X',cC)$ admits a special degeneration to $(X',cD_1)$, and hence $(X',\frac{16}{47}C)$ is K-semistable. Moreover, the pair $(X',cC)$ is K-unstable if $0<c<\frac{16}{47}$.
    \end{enumerate}
\end{prop}

\subsection{Restriction on the surfaces in the K-moduli}

In the end of this section, we explain that the surfaces of Gorenstein index $2$ in the class (iv), (v) and (vi) in Theorem \ref{40} will not appear in the K-moduli spaces $\ove{M}^K(c)$ for any $0<c<\frac{1}{2}$. As the surfaces of type (v) and (vi) degenerate to some type (iv) surfaces, then by openness of K-stability, we only need to illustrate that the type (iv) ones do not appear.

Suppose $(X,cD)$ is a K-semistable pair in the moduli stack for some $0<c<1/2$, where $X$ is a surface obtained from the blow-up of $\mb{F}_2$ along two points on a fiber $f_0$ outside the negative section $e$, and a length $2$ subscheme supported on the intersection of $f_0\cap e$. In this case, $X$ acquires a $\frac{1}{12}(1,5)$-singularity $P$.

\begin{lemma}\label{118}
    The boundary divisor $D$ must contain the singularity $P$.
\end{lemma}

\begin{proof}
    As in Section \ref{117}, the divisor $D\in |-2K_{X}|$ corresponds to a curve in on $|\mtc{O}_{\mb{P}(1,1,2)}(7)|$ defined by the polynomial $xz^3+f_3(x,y)z^2+f_5(x,y)z+f_7(x,y)$. Locally at $(0:0:1)$, the polynomial is of the form $x+f(x,y)$, where $\deg f\geq 3$. By taking the blow-up of $\mb{P}(1,1,2)$ at the cone point, we see that the local defining equation of the proper transform, denoted by $\widehat{D}$, of $\wt{D}$ on $\mb{F}_2$ is $u+ay^2+g_{\geq3}(u,y)$, where the negative section is $\{y=0\}$ and the fiber is $\{u=0\}$. Thus $\widehat{D}$ is tangent to the fiber, and hence does not pass through the tangent vector at $e\cap f_0$ we blow-up.
\end{proof}

\begin{corollary}
    The pair $(X,cD)$ is not K-semistable for any $0<c<\frac{1}{2}$.
\end{corollary}

\begin{proof}
    It follows from Lemma \ref{118} and the proof of Proposition \ref{31}.
\end{proof}

Recall that in Section \ref{102}, we list the classification of quintic $\mb{Q}$-Fano surfaces of Gorenstein index $3$. We also prove in Proposition \ref{31} that the only possible index three singularity for the surface appearing in the K-moduli space is of type $\frac{1}{9}(1,2)$. We now show that in fact for any K-semistable quintic log del Pezzo pair $(X,cD)$, the Cartier index of $X$ is at most two.

\begin{prop}\label{33}
Let $\wt{X}$ be the blow-up of the Hirzebruch surface $\mb{F}_2$ along five general points on a fiber, and $\wt{X}\rightarrow X$ be the contraction of the two negative curves. Then $(X,cD)$ is not contained in $\ove{M}^K(c)$ for any effective $\mb{Q}$-Cartier Weil divisor $D\sim_{\mb{Q}}-2K_X$ and any $0<c<1/2$.
\end{prop}

\begin{proof}
Consider the blow-up morphism $\wt{X}\rightarrow \mb{F}_2$. Let $l$ be the proper transform on $\wt{X}$ of the fiber passing through the five points, $E_i$ be the exceptional divisors over the five points $p_i$, and $\sigma$ be the negative section on $\wt{X}$. Notice that $X$ has a $\mb{G}_m$-action. If $(X,cD)\in \ove{M}^K(c)$ and $c$ is the smallest number for which the surfaces $X$ appears, then $D$ is a $\mb{G}_m$-equivariant Weil divisor. 

We have that $\Cl(X)\simeq \oplus_{i=1}^5\mb{Z}\cdot F_i$, where $F_i$ is the image of $E_i$ on $X$. Denote by $E$ and $F$ the class of a fiber and the negative section on $\mb{F}_2$. Then all the $\mb{G}_m$-equivariant prime divisors on $X$ are
\begin{enumerate}[(i)]
    \item $F_i$, where $i=1,...,5$;
    \item proper transform of a fiber, denoted by $l'$; 
    \item proper transforms of sections in the class $E+2F$ which do not pass through $p_i$, denoted by $\sigma'$;
    \item proper transforms of sections in the class $E+2F$ which pass through $p_i$, denoted by $\sigma_i$.
\end{enumerate}
We have that $l'\sim \sum F_i$, $\sigma'\sim 2\sum F_k$, and that $\sigma_i\sim (2\sum F_k)-F_i$. Denote by $\pi:\wt{X}\rightarrow X$ the contraction morphism. Then we have that $$K_{\wt{X}}\sim \pi^{*}K_X-\frac{2}{3}l-\frac{1}{3}\sigma,$$ and $$\pi^*(l')=\wt{l'}+\frac{5}{9}\sigma+\frac{1}{9}l,\quad \pi^*(F_i)=E_i+\frac{1}{9}\sigma+\frac{2}{9}l,\quad \pi^*(\sigma')=\wt{\sigma'}+\frac{1}{9}\sigma+\frac{2}{9}l, \quad \pi^{*}(\sigma_i)=\wt{\sigma_i}.$$ If $D$ is a $\mb{G}_m$-equivariant Weil divisor which is $\mb{Q}$-linearly equivalent to $-2K_X$ such that $(X,cD)$ has klt singularity, then $D$ is of one of the following type:
\begin{enumerate}[(i)]
    \item $2\sigma'+2l'$;
    \item $2\sigma'+\sigma_i+F_i$;
    \item $\sigma'+2l'+\sigma_i+F_i$;
    \item $\sigma'+\sigma_i+F_i+\sigma_k+F_k$, where we allow $i=k$;
    \item $2l'+\sigma_i+F_i+\sigma_k+F_k$, where we allow $i=k$;
    \item $\sigma_i+F_i+\sigma_j+F_j+\sigma_k+F_k$, where we allow two of $i,j,k$ coincide.
\end{enumerate}

It follows that $A_{(X,cD)}(l)\leq \frac{1-2c}{3}$. On the other hand, we compute that 
\begin{equation}
\begin{split}
    S_{(X,cD)}(l)&=\frac{1-2c}{5}\int_0^{10/3}\vol(-\pi^{*}K_X-tl)dt\\
    &=\frac{1-2c}{5}\left(\int_0^{1/3}5-\frac{9t^2}{2}dt+\int_{1/3}^{10/3}\frac{(10/3-t)^2}{2}dt\right)=\frac{11(1-2c)}{9}.
\end{split}
\end{equation} As a consequence, the pair $(X,cD)$ is never K-semistable when $0<c<1/2$ by Theorem \ref{32}.

\end{proof}

\begin{corollary}\label{56}
If $(X,cD)\in \mtc{M}^K(c)$ is a K-semistable quintic log del Pezzo pair, then the Cartier index of $X$ is at most $2$.
\end{corollary}

\section{The last divisorial contraction: plane quintic curves}

In this section, we study the last divisorial contraction $\ove{M}^K(c_i+\varepsilon)\rightarrow \ove{M}^K(c_i)$ occurring at the wall $c_i=\frac{1}{4}$. In particular, a general K-polystable pair $(X,cC)\in\ove{M}^K(c)$ on the exceptional divisor has boundary divisor $C$ a smooth plane quintic curve, which is of genus six.

\subsection{The wall $c=\frac{1}{4}$}\label{201}

Let $X$ be an ADE quintic del Pezzo surface obtained by blowing up $\mb{P}^2$ along $4$ points $p_1,...,p_4$ (not necessarily distinct). Let $q\in \mb{P}^2$ be a general point, $L_i$ be the line connecting $p_1$ and $q$, and $Q$ be the conic passing through $p_1,...,p_4,q$. Let $D_{\frac{1}{4}}\in|-2K_X|$ be the divisor obtained from the plane sextic curve $Q+L_1+\cdots+L_4$.

\begin{lemma}\label{81}
If the pair $(X,cD_{\frac{1}{4}})$ is K-semistable, then $0<c\leq\frac{1}{4}$
\end{lemma}

\begin{proof}

Let $E$ be the exceptional divisor of the blow-up of $X$ at $q$. Then we have that $A_{(X,cD_{\frac{1}{4}})}(E)=2-5c$ and that \begin{equation}
\begin{split}
    S_{(X,cD_{\frac{1}{4}})}(E)&=\frac{1-2c}{5}\int_0^{5/2}\vol(-K_{\Sigma_5}-tE)dt\\
    &=\frac{1-2c}{5}\left(\int_0^{2}5-t^2dt+\int_2^{5/2}(5-2t)^2dt\right)\\
    &=\frac{3(1-2c)}{2}.
\end{split}
\end{equation} Thus $\beta_{(X,cD_{\frac{1}{4}})}(E)\geq0$ imposes the condition $c\leq 1/4$.

\end{proof}

Take a line $l$ in the projective plane and five points $r_1,...,r_5$ on it such that at most two of them can collide. Blow up the $\mb{P}^2$ along these five points and contract the proper transform of $l$, which is a $(-4)$-curve. Here, if $q_i$ and $q_j$ are the same point, then we blow up $\mb{P}^2$ at $q_i$ along the tangent direction of $l$, and blow down the $(-2)$-exceptional curve. Then one gets a $\mb{Q}$-Fano surface $X_{q}$ with a $\frac{1}{4}(1,1)$-singularity $R$. Let $r$ be a point on the $\mb{P}^2$ outside $l$, and $l_i$ be proper transform on $X_q$ of the line $\ove{rr_i}$. Then $D_q:=\sum l_i\in|-2K_{X_q}|$ is an effective Cartier divisor.

\begin{lemma}\label{105}
The log Fano pair $(X_q,cD_q)$ is K-polystable if and only if $c=1/4$.
\end{lemma}

\begin{proof}
The pair $(X_q,cD_q)$ is a complexity one $\mb{T}$-variety, and the $\mb{G}_m$-action $\lambda$ is induced from the $\mb{G}_m$-actions on $\mb{P}^2$ defined by scaling between the point $r$ and the line $l$. Then there is no horizontal divisor on $X_q$ and all vertical divisors on $X_q$ are exceptional divisors $E_i$ and the lines joining $r$ and points on $l$.

It is easy to check that the $\beta$-invariants of the vertical divisors are positive when $c=1/4$. Let $E$ be the exceptional divisor of the blow-up of $X_q$ at $r$. Then the Futaki invariant $\Fut_{X_q}(\lambda)$ is proportional to $\beta_{X_q,cD_q}(E)$, and thus it is equal to $0$ if and only if $c=1/4$ by the same computation in the proof of Lemma \ref{81}. It follows from Theorem \ref{71} that $(X_q,cD_q)$ is K-polystable if and only if $c=\frac{1}{4}$.

\end{proof}

\begin{lemma}\label{41}
Let $X$ be a quintic del Pezzo surface with at worst $A_2$-singularities. Then the log Fano pair $(X,cD_{\frac{1}{4}})$ admits a degeneration to some log Fano pair $(X_q,cD_q)$.
\end{lemma}

\begin{proof}
We prove the statement for $X=\Sigma$, and the other two singular cases are similar. Take a trivial family $(X,D_{\frac{1}{4}})\times \mb{A}^1\rightarrow \mb{A}^1$, and blow up the total family at $(q,0)$ to get an exceptional divisor $\mb{P}^2$ over the point $0\in \mb{A}^1$, where $q\in X$ is the point with $\mult_q(D_{\frac{1}{4}})=5$. Let $\hat{X}$ be the proper transform of the central fiber $X\times\{0\}$. Running $\hat{X}$-MMP, one contracts the proper transforms of $L_1,...,L_4,Q$ successively and flips out five new curves on the exceptional divisor $\mb{P}^2$. Running $\hat{X}$-MMP once again, one contracts $\hat{X}$ and the central fiber of the family over $\mb{A}^1$ is exactly $X_q$. Tracking the boundary divisor part, one finds that the limit of $D_{\frac{1}{4}}$ is $D_q$.
\end{proof}

Now we prove our main theorem in this section.

\begin{theorem}\label{80} Let $0<c<\frac{1}{2}$ and $0<\varepsilon\ll1$ be rational numbers.
\begin{enumerate}
    \item The log Fano pair $(\Sigma_5,cD_{\frac{1}{4}})$ is K-polystable if and only if $0<c<\frac{1}{4}$. 
    \item Let $D\in |-2K_{X_q}|$ be a curve not containing the $\frac{1}{4}(1,1)$-singularity $R$, then $(X_q,cD)$ is K-semistable for $c=\frac{1}{4}$, but K-unstable for any $0<c<\frac{1}{4}$. If $D$ passes through $R$, then $(X_q,cD)$ is K-unstable for any $0<c<1$.
    \item Let $C$ be a plane quintic curve and $L$ be a line. Suppose for any point $p\in L\cap C$, we have that $p$ is a smooth point of $C$ and $\mult_p(C\cap L)\leq 2$. If $(\mb{P}^2,\frac{1}{2}C)$ is klt (e.g. $C$ is smooth or nodal), then $(X_q,c\wt{C})$ is K-stable for any $\frac{1}{4}<c<\frac{1}{2}$, where $\wt{C}$ is the proper transform of $C$ on $X_q$. 
    \item The wall-crossing morphism $\ove{M}^K(\frac{1}{4}+\varepsilon)\rightarrow \ove{M}^K(\frac{1}{4})$ is a divisorial contraction.
\end{enumerate}
\end{theorem}

\begin{proof}
 The proof is similar to proof for the first wall $c=\frac{1}{17}$.
    \begin{enumerate}
        \item By openness of the K-semistability, it follows from Lemma \ref{105} and Lemma \ref{41} that $(\Sigma_5,cD_{\frac{1}{4}})$ is K-semistable for $c=1/4$. Thus the pair $(\Sigma_5,cD_{\frac{1}{4}})$ is K-polystable for $c<1/4$ by interpolation. Using Lemma \ref{81}, one concludes that $(\Sigma_5,cD_{\frac{1}{4}})$ is K-polystable if and only if $0<c<\frac{1}{4}$.
        \item If $R\notin D$, then the $\mb{G}_m$-action on $X_q$ gives rise to an isotrivial degeneration of $(X_q,cD)$ to $(X_q,cD_q)$. As a consequence, the pair $(X_q,\frac{1}{4}D)$ is K-semistable. For $0<c<\frac{1}{4}$, the computation of $\beta$-invariant with respect to the exceptional divisor over $R$ shows that $(X_q,cD)$ is K-unstable. Similar reason explains the pair $(X_q,cD)$ is K-unstable for any $0<c<1$ if $R\in D$.
        \item The pair $(X_q,c\wt{C})$ admits a special degeneration to $(X_q,cD_q)$. Thus by openness of K-semistability, we have that $(X_q,c\wt{C})$ is K-semistable. As $(X_q,\frac{1}{2}\wt{C})$ is klt by our assumption, the interpolation theorem implies that $(X_q,c\wt{C})$ is K-stable for any $\frac{1}{4}<c<\frac{1}{2}$. Moreover, these pairs are strictly K-semistable for $c=\frac{1}{4}$ as each of them specially degenerate to a K-polystable pair $(X_q,\frac{1}{4}D_q)$.
        \item Notice that $\dim \Aut(X_q)=3$ and there is a two-dimensional moduli space of such surfaces. As the pairs considered in (2) are all strictly K-semistable, then they lie in the exceptional locus of the wall-crossing morphism $$\phi^{+}_{\frac{1}{4}}:\ove{M}^K\left(1/4+\varepsilon\right)\rightarrow \ove{M}^K\left(1/4\right).$$ In particular, the exceptional locus is $14$-dimensional, and $\phi^{+}_{\frac{1}{4}}$ is a divisorial contraction.
    \end{enumerate}
\end{proof}

Let $E_3(c)$ be the sublocus of $\ove{M}^K(c)$ parameterizing pairs admitting a smoothing to the pair $(X_q,cD)$, where $D\in|-2K_{X_q}|$ is a smooth curve. In the construction of $X_q$, we require that at most two of the points to blow up can collide. There is no reason to not allow more points to collide. If there is a K-semistable pair $(X'_q,cD')$ for some $0<c<\frac{1}{2}$, then this pair is on the exceptional divisor $E_3(c)$. In the next section, we will find out all the walls corresponding such pairs and illustrate the relation to Laza's VGIT moduli spaces of $(1,5)$-curves on $\mb{P}^2$.

\subsection{Walls on the exceptional divisor and Laza's VGIT for $(1,5)$-curves on $\mb{P}^2$}

Let $(X_q,cC)$ be a pair, where $X_q$ is from blowing up $\mb{P}^2_{(x:y:z)}$ along five collinear points (not necessarily distinct), and $C\in |-2K_{X_q}|$ is a curve. We may assume that the five points to blow up is on the line $\{x=0\}$. Then the curve $C$ corresponds uniquely to a plane quintic curve $\ove{C}$ given by a polynomial $xf_4(x,y,z)+f_5(y,z)$.

Applying the same argument to getting the Table \ref{Kwall4}, one can find all the walls on the exceptional divisor $E_3(c)$, listed in the following table.

\begin{center}
\renewcommand*{\arraystretch}{1.2}
\begin{table}[ht]
    \centering
      \begin{tabular}{ |c  |c |c|c|}
    \hline
     wall & curves & weight of $\mb{G}_m$-action   \\ \hline 
     
     $\frac{19}{68} $  &  $xz^4+y^3z^2=0$  & $(3,1,0)$  
     \\ \hline

    $\frac{23}{76} $  &  $xyz^3+y^3z^2=0$  & $(2,1,0)$  
     \\ \hline 
     
     $\frac{9}{28} $  &  $xz^4+y^4z=0$  & $(4,1,0)$  
     \\ \hline     

     $\frac{31}{92} $  &  $x^2z^3+y^3z^2=0$  & $(3,2,0)$  
     \\ \hline 
     
     $\frac{7}{20} $  &  $xz^4+y^5=0$  & $(5,1,0)$  
     \\ \hline     
  
    $\frac{13}{36} $  &  $xyz^3+y^4z=0$  & $(3,1,0)$  
     \\ \hline  
     
      $\frac{11}{28} $  &  $xyz^3+y^5=0$  & $(4,1,0)$  
     \\ \hline  
  
   \end{tabular}
    \caption{Walls corresponding to $X_{q}$}
    \label{Kwall5}
\end{table}
\end{center}

In \cite{Laz07}, the author studies the VGIT moduli spaces $$\ove{M}^{\GIT}_{\mb{P}^2}(t):=\left(|\mtc{O}_{\mb{P}^2}(5)|\times|\mtc{O}_{\mb{P}^2}(1)|\right)\sslash_{\mtc{O}(1,t)}\PGL(3)$$ of $(1,5)$-plane curve pairs. See \cite[Section 3]{Laz07} for details on the stability conditions for different $0<t\leq\frac{5}{2}$. Our main theorem in this section is the following.

\begin{theorem}
    Let ${E}_{3}(c)$ be the closed subscheme of $\ove{M}^K(c)$ parameterizing pairs $(X,cD)$ admitting a partial smoothing to a pair $(X_q,cC)$, where $X_q$ is obtained from blowing up $\mb{P}^2$ along five distinct collinear points, and $C\in|-2K_{X_q}|$ is a smooth curve. Then $E_3(c)$ is isomorphic to Laza's VGIT moduli $\ove{M}^{\GIT}_{\mb{P}^2}(t)$ of $(1,5)$-curves on $\mb{P}^2$ under the relation $t=\frac{25-20c}{28c+1}$.
\end{theorem}

\begin{proof}
    By the classification of surfaces appearing in the K-moduli, we see that $X$ must be obtained from blowing up $\mb{P}^2$ along five collinear points, not necessarily distinct, then contracting the $(-4)$-curve. As $(X_q,cD)$ is K-semistable, the curve $D$ avoid the $\frac{1}{4}(1,1)$-singularity $R$. Blowing up $X$ at $R$ and contracting the five exceptional curves in order, we obtain a surface pair $(C,L)$, where $C$ is a quintic curve and $L$ is a line, and $C$ does not have $L$ as a component. Conversely, if $(C,L)\in |\mtc{O}_{\mb{P}^2}(5)|\times |\mtc{O}_{\mb{P}^2}(1)|$ is a pair such that $C$ does not have $L$ as a component, then blowing up $C\cap L$ and contracting the strict transform of $L$ yields a surface $X$ which admits a partial smoothing to $X_q$, and the birational transform of $C$ on $X$ is of the class $-2K_X$.

    By comparing the stability conditions in the Table \ref{Kwall5} and the stability conditions in \cite[Table 4]{Laz07}, we deduce that there exists a morphism from the moduli stack $\mtc{M}^{\GIT}_{\mb{P}^2}(c)\rightarrow\mtc{E}_3(c)$, which descends to a bijective morphism  $f(c):\ove{M}^{\GIT}_{\mb{P}^2}(t(c))\rightarrow E_3(c)$ between good moduli spaces, under the relation $t(c)=\frac{25-20c}{28c+1}$. The same argument as in Proof below to Theorem \ref{6} shows that the deformation of $X$ is unobstructed, and hence $\mtc{E}_3(c)$ is smooth. In particular, the good moduli space $E_3(c)$ is normal. By Zariski's main theorem and the normality of $E_3(c)$, we conclude that $f(c)$ is an isomorphism. Here $\frac{1}{4}\leq c<\frac{1}{2}$ and $1\leq t(c)\leq \frac{5}{2}$.
\end{proof}

\subsection{Proof of Theorem \ref{6}}

Now we can complete our proof of Theorem \ref{6}.

\begin{proof}\label{202}
    It remains to show that the moduli stack $\mtc{M}^K(c)$ is smooth, and that the good moduli space $\ove{M}^K(c)$ is irreducible and normal.

    We first prove that for any K-semistable pair $(X,cC)\in\mtc{M}^K(c)$, there are no local-to-global obstructions for deformations. By \cite[Section 3.4.4]{Se06}, the obstructions lie in the space $\Ext^2(\Omega^1_X(\log(C_{\red}) ,\mtc{O}_X))$. As $X$ is $\mb{Q}$-Fano, then we have that $\Ext^2(\Omega^1_X,\mtc{O}_X)=0$ (cf. \cite[Proposition 3.10]{HP10}). Thus the vanishing of $\Ext^2(\Omega^1_X(\log(C_{\red}),\mtc{O}_X))$ follows from the exact sequence $$0\longrightarrow \Omega^1_X\longrightarrow\Omega^1_X(\log (C_{\red}))\longrightarrow \bigoplus \mtc{O}_{C_i}\longrightarrow 0,$$ where $C_i$ are irreducible components of $C_{\red}$.

    For any K-semistable pair $(X,cD)\in\mtc{M}^K(c)$, it follows from \cite[Theorem 3.33]{ADL19} that there is an \'{e}tale morphism $$\left[\Ext^1(\Omega^1_X(\log(C_{\red}),\mtc{O}_X))\sslash \Aut(X,D)\right]\rightarrow \mtc{M}^K(c)$$ locally at $[(X,cD)]$, and hence $\mtc{M}^K(c)$ is smooth (cf. \cite[Tag 0DLS]{SP}). The normality of $\ove{M}^K(c)$ follows from the same argument with the quotient stack replaced by GIT quotient. The irreducibility of $\ove{M}^K(c)$ follows from the fact that $\Sigma_5$ is the unique smooth Fano surface with anti-canonical degree $5$ and the argument in \cite[Proposition 3.6]{ADL22}.
\end{proof}

\section{Relation to moduli of K3 surfaces and moduli of curves}

In this section, we will that there are no other divisorial contractions and that the Picard rank of $\ove{M}^K\left(\frac{1}{2}-\varepsilon\right)$ is $4$ for $0<\varepsilon\ll1$.

\subsection{Kond\={o}'s moduli of K3 surfaces}

Let us first recall some constructions and results in moduli spaces of K3 surfaces. For any $\mb{Q}$-Fano surface $X$ of degree $5$ and a curve $C\in |-2K_X|$, taking double cover of $X$ branched along $C$ yields a K3 surface $Y$. The Picard lattice of $Y$ contains a rank 5 lattice, whose orthogonal complement in $H^2(Y;\mb{Z})$ is denoted by $T$. Then we consider the period domain $$\mb{D}=\{\omega\in\mb{P}(T_{\mb{C}}):(\omega,\omega)=0, \quad (\omega,\ove{\omega})>0)\}.$$ Let $\Gamma$ be the orthogonal group of $T$, and $\mtc{F}:=\mb{D}/\Gamma$ the arithmetic quotient. The Picard number of $\mtc{F}$ is equal to $4$, and the four generators corresponding to the curves of genus $6$ with less than five $g^2_6$, nodal curves, trigonal curves, and plane quintic curves (cf. \cite[Theorem 0.2]{AK11}). Let $\mtc{F}^{*}$ be the Satake-Baily-Borel compactification of $\mtc{F}$. Then the boundary of  $\mtc{F}^{*}$ has 2 zero-dimensional components and 14 one-dimensional components (cf. \cite[Theorem 0.3]{AK11}).

\begin{theorem}\label{93} Let $0<\varepsilon\ll1$ be a rational number.
\begin{enumerate}
    \item The moduli space $\ove{M}^K\left(\frac{1}{2}-\varepsilon\right)$ has Picard number $4$. 
    \item The moduli space $\mtc{F}^{*}$ is the ample model of $\ove{M}^K\left(\frac{1}{2}-\varepsilon\right)$ with respect to the Hodge $\mb{Q}$-line bundle $\lambda_{\Hodge}$.
\end{enumerate}
   
\end{theorem}

\begin{proof}\*
Recall that we classified all the surfaces that appear in our moduli space: six types of ADE del Pezzo surfaces and three types of surfaces of Gorenstein index $2$. We also computed the dimensions of the exceptional loci of the wall crossing morphisms $\ove{M}^K(c+\varepsilon)\rightarrow \ove{M}^K(c)$: there are three divisorial contractions $\phi_{c_i}^{+}:\ove{M}^K(c_i+\varepsilon)\rightarrow \ove{M}^K(c_i)$, and the three divisors will not be contracted during further wall-crossings. Therefore the moduli space $\ove{M}^K\left(\frac{1}{2}-\varepsilon\right)$ has Picard number $4$ as desired.

For a log Fano pair $(X,cD)\in \ove{M}^K\left(\frac{1}{2}-\varepsilon\right)$, by taking the double cover of $X$ branched along $D$, we get a rational map $$\ove{M}^K\left(\frac{1}{2}-\varepsilon\right)\dashrightarrow \mtc{F}^{*}$$ which is an isomorphism in codimension 1 by \cite[Section 3]{AK11} and (1). Precisely, there is a big open subset $U$ of $\ove{M}^K\left(\frac{1}{2}-\varepsilon\right)$ mapped isomorphically to $V\subseteq \mtc{F}$. By \cite[Section 6.2]{Huy16}, the restriction of the Hodge line bundle $\lambda_{\Hodge}$ on $U$ is ample because it is the pull back of the Hodge line bundle on $\mtc{F}$. As $\mtc{F}^{*}$ is the Satake-Baily-Borel compactification, then this line bundle extends to an ample line bundle on $\mtc{F}^{*}$. On the other hand, it follows from \cite[Theorem 9.7]{CP21} that the Hodge $\mb{Q}$-line bundle $\lambda_{\Hodge}$ is nef. Since $U\subseteq \ove{M}^{K}\left(\frac{1}{2}-\varepsilon\right)$ is a big open subset, then $\lambda_{\Hodge}$ is big and thus semi-ample on $\ove{M}^K\left(\frac{1}{2}-\varepsilon\right)$. As a consequence, the moduli space $$\mtc{F}^{*}=\Proj\left(
    \bigoplus_{m\geq0}H^0\left(\ove{M}^K\left(1/2-\varepsilon\right),\lambda_{\Hodge}^{\otimes m}\right)\right)$$ is the ample model of $\lambda_{\Hodge}$.
\end{proof}

\subsection{Birational map to the moduli of curves}

We have the following results.

\begin{theorem}\label{52}
Let $0<\varepsilon\ll$ be a rational number, and $\varphi(c):\ove{M}^K(c)\dashrightarrow \ove{M}_6$ be the natural forgetful map defined by $[(X,cC)]\mapsto [C]$.
\begin{enumerate}
    \item The map $\varphi(\varepsilon)$ is surjective onto $U$ and dominates $\Delta_{0}$.
    \item The rational map $\varphi\left(\frac{1}{2}-\varepsilon\right)$ restricts to a surjective morphism $$M_{\sm}^K(1/2-\varepsilon)\twoheadrightarrow M_6\setminus (Z_h\cup Z_b),$$ where $M_{\sm}^K(c)$ denotes the open subscheme of $M^K(c)$ consisting of the K-polystable pairs $(X,cD)$ with $D$ smooth.
\end{enumerate}
\end{theorem}

\begin{proof}
The first statement follows from Theorem \ref{55}. Indeed, for each pair $(\Sigma,\varepsilon C)\in \ove{M}^K(\varepsilon)$, if $C$ is smooth, then $C$ has exactly five distinct linear series $g^2_6$. There is a divisor $E_0(\varepsilon)$ in $\ove{M}^K(\varepsilon)$ parameterizing pairs $(\Sigma,\varepsilon C)$ where $C$ is a singular curve. A general such singular curve $C$ has at worst nodal singularities, and hence is an element in $\Delta_0$. Since $C$ is canonically embedded in $\Sigma_5$, then the map $E_0(\varepsilon)\dashrightarrow \Delta_0$ is generically finite, and hence dominant by counting dimensions. 

For (2), we first notice that by the description of walls and K-semistable elements in the previous three sections, if $(X,cC)\in\ove{M}^K(c)$ is a K-polystable element and $C$ is smooth, then $C$ is either a non-special curve, or a plane quintic curve, or a trigonal curve. Conversely, let $C$ be a smooth curve of genus six which is neither hyperelliptic nor bielliptic. If $C$ has finitely many $g^2_6$, then there exists a unique ADE quintic del Pezzo surface $X$ such that $(X,cC)$ is K-polystable for any $\frac{11}{28}<c<\frac{1}{2}$ by Remark \ref{203}. Similarly, if $C$ is trigonal or a plane quintic curve, then there is a surface $X=X_t$ or $X=X_q$ such that $(X,cC)$ is K-polystable for $\frac{11}{28}<c<\frac{1}{2}$. This proves (2).
\end{proof}

\subsection{Further discussion}

In the end, let us give some explanations on the loci of hyperelliptic and bielliptic curves.

In \cite{Yan96}, the author gives a complete classification of plane sextic curves with ADE singularities. In \cite[Table 2]{Yan96} that there is a type of curves with an $A_{13}$-singularity and 4 nodes in general position. Blowing up the plane at these four points will send these curves into $|-2K_{\Sigma_5}|$, whose stable reduction is a smooth genus six hyperelliptic curve. By \cite[6.2.1]{Has00} we know that all genus 6 smooth hyperelliptic curves arises this way. See \cite[Section 4.5]{Gol21}  for a similar description. Notice that such curves have log canonical threshold $\lct(X,C)=\frac{1}{2}+\frac{1}{14}$, thus we cannot see hyperelliptic curves in the K-moduli $\ove{M}^K(c)$ for $0<c<1/2$.

Consider the pair $(\Sigma_5,cC)$, where $C$ is the proper transform of a double elliptic curve on $\mb{P}^2$ passing through the four blown-up points. Such a pair is K-stable for $0<c<1/2$. When $c=1/2$, the pair will degenerate to a cone over an elliptic curve together with a double elliptic curve at infinity. Across this wall, the bielliptic curves will appear. Furthermore, any bielliptic curve of genus six can be realized as a quadric section of a cone in $\mb{P}^5$ over a elliptic normal curve in $\mb{P}^4$ which avoids the vertex.

For curves in $\Delta_i$ for $i=2,3$, we expect that they appear in the KSBA-moduli spaces of stable pairs $(X,cD)$ for $1/2<c<1$. In KSBA-moduli spaces, we allow the pair to have slc singularities. In particular, the surfaces can be reducible. Thus for each such a reducible surface $X$, every section in the linear series $|-2K_X|$ is reducible.

\bibliographystyle{alpha}
\bibliography{citation}

\end{document}